\documentclass[reqno,letterpaper,11pt]{article}

\usepackage{hyperref}

\usepackage{palatino}

\usepackage{amsmath}
\usepackage{amsfonts,amssymb,amsmath,latexsym,wasysym,mathrsfs,bbm,stmaryrd}
\usepackage[all]{xy}
\usepackage[usenames]{color}
\usepackage{epsfig}
\usepackage[mathcal]{euscript}
\usepackage{graphics,moreverb,gastex}

\usepackage{amsthm}
\usepackage{thmtools}
\usepackage{thm-restate}
\declaretheorem[numberwithin=section]{theorem}
\declaretheorem[sibling=theorem]{proposition}
\declaretheorem[sibling=theorem]{definition}
\declaretheorem[sibling=theorem]{corollary}
\declaretheorem[sibling=theorem]{lemma}

\declaretheorem[sibling=theorem,style=remark]{remark}
\declaretheorem[sibling=theorem,style=remark]{example}

\numberwithin{equation}{section} 

\def\R{\mathbb R}
\def\C{\mathbb C}
\def\Z{\mathbb Z}
\def\N{\mathbb N}

\def\P{\mathbb P}
\def\E{\mathbb E}
\def\Var{\mathrm{Var}}

\def\u{\mathfrak{u}}

\def\1{\mathbbm 1}

\def\e{\epsilon}
\def\d{\delta}

\def\t{\tau}

\def\M{\mathbb{M}}
\def\U{\mathbb{U}}
\def\GL{\mathbb{GL}}
\def\A{\mathscr{A}}

\def\u{\mathfrak{u}}

\def\tensor{\otimes}
\def\supp{\mathrm{supp}\,}
\def\del{\partial}

\newcommand{\tr}{\mathrm{tr}}
\newcommand{\Tr}{\mathrm{Tr}}
\newcommand{\mx}[1]{\mathbf{#1}}

\renewcommand\emptyset{\varnothing}

\def\t{\tau}

\long\def\symbolfootnote[#1]#2{\begingroup%
\def\thefootnote{\fnsymbol{footnote}}\footnote[#1]{#2}\endgroup}

\begin{document}

\title{Strong Convergence of Unitary Brownian Motion}
\author{
  Beno\^\i{}t Collins\thanks{Supported in part by Kakenhi, an NSERC discovery grant, Ontario's
  ERA and ANR-14-CE25-0003} \\
  Kyoto University \\
  \& University of Ottawa \&  CNRS Lyon I\\
  585 King Edward, Ottawa, ON K1N 6N5 \\
  \texttt{bcollins@uottawa.ca}
\and
Antoine Dahlqvist \\
Technische Universit\"at Berlin, Fakult\"at II  \\
Institut f\"ur Mathematik, Sekr. MA 7-5 \\
Strasse des 17. Juni 136, 10623 Berlin, Germany \\
\texttt{antoine.dahlqvist@gmail.com}
\and
Todd Kemp\thanks{Supported in part by NSF CAREER Award DMS-1254807} \\
Department of Mathematics \\
University of California, San Diego \\
La Jolla, CA 92093-0112 \\
\texttt{tkemp@math.ucsd.edu}
}

\date{\today} 

\maketitle

\begin{abstract} The Brownian motion $(U^N_t)_{t\ge 0}$ on the unitary group converges, as a process, to the free unitary Brownian motion $(u_t)_{t\ge 0}$ as $N\to\infty$.  In this paper, we prove that it converges {\em strongly} as a process: not only in distribution but also in operator norm.  In particular, for a fixed time $t>0$, we prove that the spectral measure has a {\em hard edge}: there are no outlier eigenvalues in the limit.  We also prove an extension theorem: any strongly convergent collection of random matrix ensembles independent from a unitary Brownian motion also converge strongly {\em jointly} with the Brownian motion.  We give an application of this strong convergence to the Jacobi process. \end{abstract}

\tableofcontents

\section{Introduction\label{section Introduction}}

This paper is concerned with convergence of the noncommutative distribution of the standard Brownian motion on unitary groups.  Let $\M_N$ denote the space of $N\times N$ complex matrices, and let $\tr = \frac1N\Tr$ denote the normalized trace.  We denote the unitary group in $\M_N$ as $\U_N$.  The Brownian motion on $\U_N$ is the diffusion process $(U^N_t)_{t\ge 0}$ started at the identity with infinitesimal generator $\frac12\Delta_{\U_N}$, where $\Delta_{\U_N}$ is the left-invariant Laplacian on $\U_N$.  (This is uniquely defined up to a choice of $N$-dependent scale; see Section \ref{section UBM def} for precise definitions, notation, and discussion.)

For each fixed $t\ge 0$, $U^N_t$ is a random unitary matrix, whose spectrum $\mathrm{spec}(U^N_t)$ consists of $N$ eigenvalues $\lambda_1(U^N_t),\ldots,\lambda_N(U^N_t)$.  The {\em empirical spectral distribution}, also known as the {\em empirical law of eigenvalues}, of $U^N_t$ (for a fixed $t\ge0$) is the random probability measure $\mathrm{Law}_{U^N_t}$ on the unit circle $\U_1$ that puts equal mass on each eigenvalue (counted according to multiplicity):
\[ \mathrm{Law}_{U^N_t} = \frac1N\sum_{j=1}^N \delta_{\lambda_j(U^N_t)}. \]
In other words: $\mathrm{Law}_{\U^N_t}$ is the random measure determined by the characterization that its integral against test functions $f\in C(\U_1)$ is given by
\begin{equation} \label{e.Law.UN} \int_{\U_1} f\,d\mathrm{Law}_{U^N_t} = \frac1N\sum_{j=1}^N f(\lambda_j(U^N_t)). \end{equation}
In \cite{Biane1997a}, Biane showed that this random measure converges weakly almost surely to a deterministic limit probability measure $\nu_t$
\begin{equation} \label{e.nu_t=limit} \lim_{N\to \infty}\int_{U_1} f\,d\mathrm{Law}_{U^N_t} = \int_{\U_1} f\,d\nu_t \; a.s. \qquad f\in C(\U_1). \end{equation}
The measure $\nu_t$ can be described as the spectral measure of a {\em free unitary Brownian motion} (cf.\ Section \ref{section free prob}).  For $t>0$, $\nu_t$ possesses a continuous density that is symmetric about $1\in\U_1$, and is supported on an arc strictly contained in the circle for $0<t<4$; for $t\ge 4$, $\supp \nu_t = \U_1$.

The result of \eqref{e.nu_t=limit} is a bulk result: it does not constrain the behavior of individual eigenvalues.  The additive counterpart is the classical Wigner's semicircle law.  Let  $X^N$ is a Gaussian unitary ensemble ($\mathrm{GUE}^N$), meaning that the joint density of entries of $X^N$ is proportional to $\exp\left(-\frac{N}{2}\Tr(X^2)\right)$. Alternatively, $X^N$ may be described as a {\em Gaussian Wigner matrix}, meaning it is Hermitian, and otherwise has i.i.d.\ centered Gaussian entries of variance $\frac1N$.  Wigner's law states that the empirical spectral distribution converges weakly almost surely to a limit: the semicircle distribution $\frac{1}{2\pi}\sqrt{(4-x^2)_+}\,dx$, supported on $[-2,2]$ (cf.\ \cite{Wigner}).  This holds for all Wigner matrices, independent of the distribution of the entries, cf.\ \cite{Bai1999}.  But this does not imply that the spectrum of $X^N$ converges almost surely to $[-2,2]$; indeed, it is known that this ``hard edge'' phenomenon occurs iff the fourth moments of the entries of $X^N$ are finite (cf.\ \cite{BaiYin1988}).

Our first major theorem is a ``hard edge'' result for the empirical law of eigenvalues of the Brownian motion $\U^N_t$.  Since the spectrum is contained in the circle $\U_1$ instead of discussion the ``largest'' eigenvalue, we characterize convergence in terms of Hausdorff distance $d_H$: the Hausdorff distance between two compact subsets $A,B$ of a metric space is defined to be 
\[ d_H(A,B) = \inf\{\e\ge 0\colon A\subseteq B_\e \;\&\; B\subseteq A_\e\}, \]
where $A_\e$ is the set of points within distance $\e$ of $A$. It is easy to check that the ``hard edge'' theorem for Wigner ensembles is equivalent to the statement that $d_H(\mathrm{spec}(X^N),[-2,2])\to 0 \; a.s.$ as $N\to\infty$; for a related discussion, see Corollary \ref{c.strong.conv} and Remark \ref{r.strong.conv} below.

\begin{theorem} \label{t.main.1} Let $N\in\N$, and let $(U^N_t)_{t\ge 0}$ be a Brownian motion on $\U_N$.  Fix $t\ge 0$.  Denote by $\nu_t$ the law of the free unitary Brownian motion, cf.\ Theorem \ref{t.nu.t}.  Then
\[ d_H\left(\mathrm{spec}(U^N_t),\supp\nu_t\right)\to 0\;\; a.s. \;\text{as}\; N\to\infty. \]
\end{theorem}

\begin{remark} When $t\ge 4$, $\supp\nu_t = \U_1$, and Theorem \ref{t.main.1} is immediate; the content here is that, for $0\le t<4$, there are asymptotically {\em no eigenvalues} outside the arc in \eqref{e.supp.nu_t}.
\end{remark}

\begin{figure}[htbp]
\begin{center}
\includegraphics[scale=0.1965]{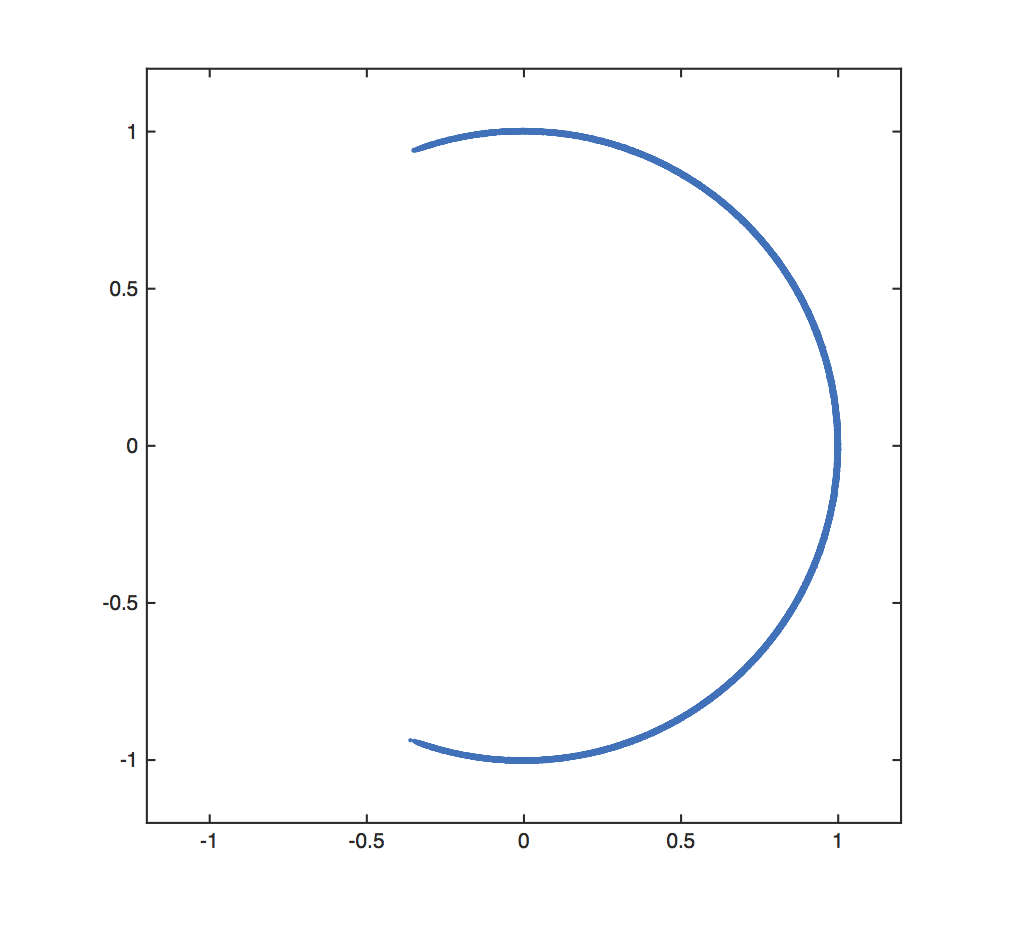}
\includegraphics[scale=0.208]{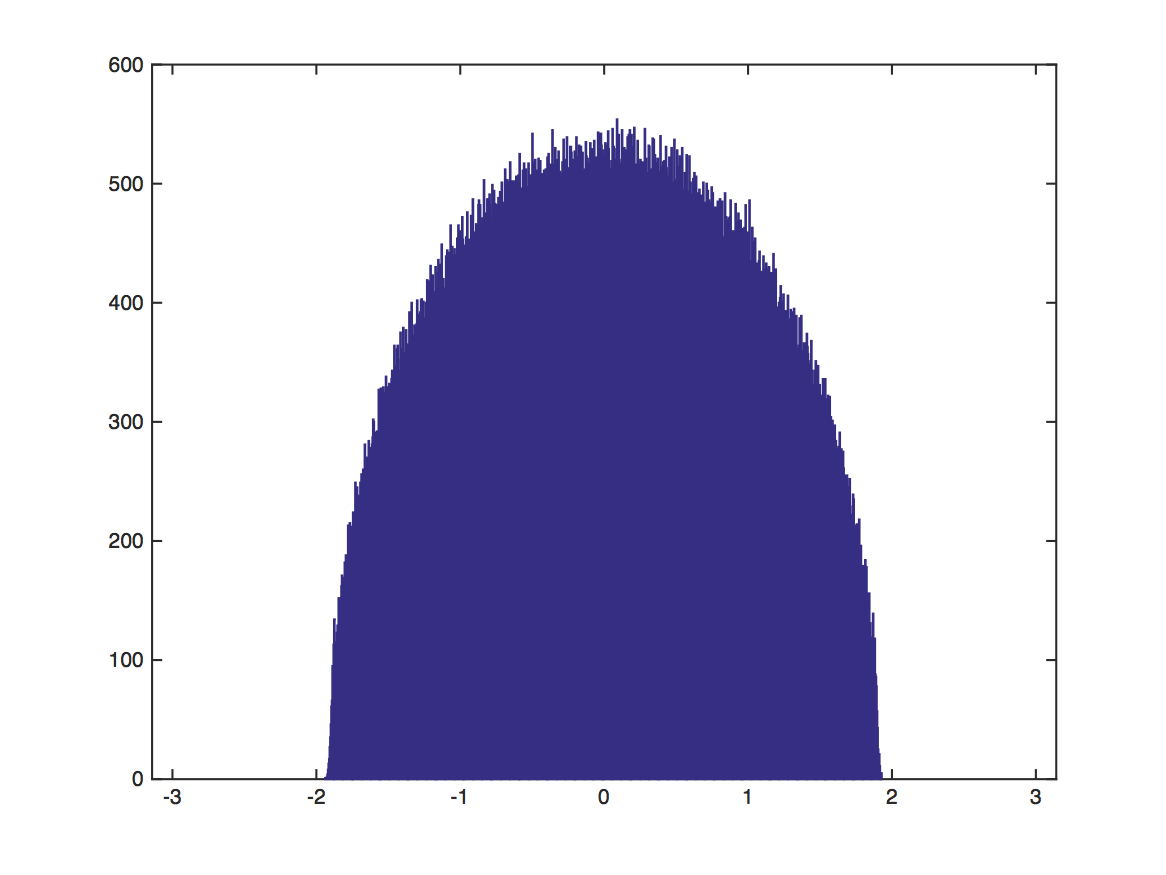}
\caption{\label{fig spec UBM} The spectrum of the unitary Brownian motion $U^N_t$ with $N=400$ and $t=1$.  These figures were produced from $1000$ trials.  On the left is a plot of the eigenvalues, while on the right is a $1000$-bin histogram of their complex arguments.  The argument range of the data is $[-1.9392,1.9291]$, as compared to the predicted large-$N$ limit range (to four digits) $[-1.9132,1.9132]$, cf.\ \eqref{e.supp.nu_t}.}
\end{center}
\end{figure}

To prove Theorem \ref{t.main.1}, our method is to prove sufficiently tight estimates on the rate of convergence of the moments of $U^N_t$.  We record the main estimate here, since it is  of independent interest.

\begin{theorem} \label{t.moment.rate.conv} Let $N,n\in\N$, and fix $t\ge 0$.  Then
\begin{equation} \label{e.moment.rate.conv} \left|\E\tr\!\big[(U^N_t)^n\big]- \int_{\U_1} w^n\,\nu_t(dw)\right| \le \frac{t^2n^4}{N^2}. \end{equation}
\end{theorem}
\noindent The tool we use to calculate these moments is the Schur-Weyl duality for the representation theory of $\U_N$; see Section \ref{section Schur-Weyl Duality}.  Theorems \ref{t.main.1} and  \ref{t.moment.rate.conv} are proved in Section \ref{section hard edge}.

The second half of this paper is devoted to a multi-time, multi-matrix extension of this result.  Biane's main theorem in \cite{Biane1997a} states that the {\em process} $(U^N_t)_{t\ge 0}$ converges (in the sense of finite-dimensional noncommutative distributions) to a free unitary Brownian motion $u=(u_t)_{t\ge 0}$.  To be precise: for any $k\in\N$ and times $t_1,\ldots,t_k\ge 0$, and any noncommutative polynomial $P\in\C\langle X_1,\ldots,X_{2k}\rangle$ in $2k$ indeterminates, the random trace moments of $(U^N_{t_j})_{1\le j\le k}$ converge almost surely to the corresponding trace moments of $(u_{t_j})_{1\le j\le k}$:
\[ \lim_{N\to\infty} \tr\!\left(P(U^N_{t_1},(U^N_{t_1})^\ast,\ldots,U^N_{t_k},(U^N_{t_k})^\ast)\right) = \t\!\left(P(u_{t_1},u_{t_1}^\ast,\ldots,u_{t_k},u_{t_k}^\ast)\right) \;\; a.s.  \]
(Here $\t$ is the tracial state on the noncommutative probability space where $(u_t)_{t\ge 0}$ lives; cf.\ Section \ref{section free prob}.)  This is the noncommutative extension of a.s.\ weak convergence of the empirical spectral distribution.  The corresponding strengthening to the level of the ``hard edge'' is {\em strong convergence}: instead of measuring moments with the linear functionals $\tr$ and $\t$, we insist on a.s.\ convergence of polynomials in {\em operator norm}.  See Section \ref{section nc dist} for a full definition and history.

Theorem \ref{t.main.1} can be rephrased to say that, for any fixed $t\ge 0$, $(U^N_t,(U^N_t)^\ast)$ converges strongly to $(u_t,u_t^\ast)$ (cf.\ Corollary \ref{c.strong.conv}).  Our second main theorem is the extension of this to any finite collection of times.  In fact, we prove a more general extension theorem, as follows.

\begin{theorem} \label{t.main.2} For each $N$, let $(U^N_t)_{t\ge 0}$ be a Brownian motion on $\U_N$.  Let $A^N_1,\ldots,A^N_n$ be random matrix ensembles in $\M_N$ all independent from $(U^N_t)_{t\ge 0}$, and suppose that $(A^N_1,\ldots,A^N_n)$ converges strongly to $(a_1,\ldots,a_n)$.  Let $(u_t)_{t\ge 0}$ be a free unitary Brownian motion freely independent from $\{a_1,\ldots,a_n\}$.  Then, for any $k\in\N$, and any $t_1,\ldots,t_k\ge 0$,
\[ (A^N_1,\ldots,A^N_n,U^N_{t_1},\ldots,U^N_{t_k}) \;\;\text{converges strongly to}\;\; (a_1,\ldots,a_n,u_{t_1},\ldots,u_{t_k}). \]
\end{theorem}
\noindent Theorem \ref{t.main.2} is proved in Section \ref{section strong convergence}.

We conclude the paper with an application of these strong convergence results to the empirical spectral distribution of the Jacobi process, in Theorem \ref{t.final}.  We proceed now with Section \ref{section Background}, laying out the basic concepts, preceding results, and notation we will use throughout.

\newpage

\section{Background\label{section Background}}

Here we set notation and briefly recall some main ideas and results we will need to prove our main results.  Section \ref{section UBM def} introduced the Brownian motion $(U^N_t)_{t\ge 0}$ on $\U_N$. Section \ref{section Schur-Weyl Duality} discusses the Schur-Weyl duality for the representation theory of $\U_N$, and its use in computing expectations of polynomials in $U^N_t$.  Section \ref{section nc dist} discusses noncommutative distributions (which generalize empirical spectral distributions to collections of noncommuting random matrix ensembles, and beyond) and associated notions of convergence, including strong convergence.  Finally, Section \ref{section free prob} reviews key ideas from free probability and free stochastic calculus, leading up to the definition of free unitary Brownian motion and its spectral measure $\nu_t$.

\subsection{Brownian Motion on $\U_N$ \label{section UBM def}}

Throughout, $\U_N$ denotes the unitary group of rank $N$; its Lie algebra $\mathrm{Lie}(\U_N) = \u_N$ consists of the skew-Hermitian matrices in $\M_N$, $\u_N = \{X\in\M_N\colon X^\ast=-X\}$.  We define a real inner product on $\u_N$ by scaling the Hilbert-Schmidt inner product
\[ \langle X,Y\rangle_N \equiv -N\Tr(XY), \qquad X,Y\in\u_N. \]
As explained in \cite{DHK2013}, this is the unique scaling that gives a meaningful limit as $N\to\infty$.

Any vector $X\in\u_N$ gives rise to a unique left-invariant vector field on $\U_N$; we denote this vector field as $\del_X$ (it is more commonly called $\widetilde{X}$ in the geometry literature).  That is: $\del_X$ is a left-invariant derivation on $C^\infty(\U_N)$ whose action is
\[ (\del_Xf)(U) = \left.\frac{d}{dt}\right|_{t=0} f(Ue^{tX}) \]
where $e^{tX}$ denotes the usual matrix exponential (which is the exponential map for the linear Lie group $\U_N$; in particular $e^{tX}\in\U_N$ whenever $X\in\u_N$).  The {\bf Laplacian} $\Delta_{\U_N}$ on $\U_N$ (determined by the metric $\langle\cdot,\cdot\rangle_N$) is the second-order differential operator
\[ \Delta_{\U_N} \equiv \sum_{X\in\beta_N} \del_X^2 \]
where $\beta_N$ is any orthonormal basis for $\u_N$; the operator does not depend on which orthonormal basis is used.  The Laplacian is a negative definite elliptic operator; it is essentially self-adjoint in $L^2(\U_N)$ taken with respect to the Haar measure (cf. \cite{Robinson1991,Varopoulos1992}).

The {\bf unitary Brownian motion} $U^N = (U^N_t)_{t\ge 0}$ is the Markov diffusion process on $\U_N$ with generator $\frac12\Delta_{\U_N}$, with $U^N_0=I_N$.  In particular, this means that the law of $U^N_t$ at any fixed time $t\ge 0$ is the {\bf heat kernel measure} on $\U_N$.  This is essentially by definition: the heat kernel measure $\rho^N_t$ is defined weakly by
\begin{equation} \label{e.expectation.1} \E_{\rho^N_t}(f) \equiv \int_{\U_N} f\,d\rho_t^N = \left(e^{\frac{t}{2}\Delta_{\U_N}}f\right)(I_N), \qquad f\in C(\U_N). \end{equation}
There are many tools for computing the heat kernel using the representation theory of $\U_N$; we discuss this further in Section \ref{section Schur-Weyl Duality}.  We mention here the fact that the heat kernel is symmetric: it is invariant under $U\mapsto U^{-1}$ (this is true on any Lie group).

There are (at least) two more constructive ways to understand the Brownian motion $U^N$  directly.  The first is as a L\'evy process: $U^N$ is uniquely defined by the properties
\begin{itemize} \label{U^N invariance}
\item {\sc Continuity}: The paths $t\mapsto U^N_t$ are a.s.\ continuous.
\item {\sc Independent Multiplicative Increments}: For $0\le s\le t$, the multiplicative increment $(U_s^N)^{-1}U_t^N$ is independent from the filtration up to time $s$ (i.e.\ from all random variables measurable with respect to the entires of $U_r^N$ for $0\le r\le s$).
\item {\sc Stationary Heat-Kernel Distributed Increments}: For $0\le s\le t$, the multiplicative increment $(U_s^N)^{-1}U_t^N$ has the distribution $\rho^N_{t-s}$.
\end{itemize}
In particular, since $U^N_t$ is distributed according to $\rho^N_t$, we typically write expectations of functions on $\U_N$ with respect to $\rho^N_t$ as
\[ \E_{\rho^N_t}(f) = \E[f(U^N_t)]. \]

For the purpose of computations, the best representation of $U^N$ is as the solution to a stochastic differential equation.  Let $X^N$ be a $\mathrm{GUE}^N$-valued Brownian motion: that is, $X^N$ is Hermitian where the random variables $[X^N]_{jj}, \Re[X^N]_{jk}, \Im[X^N]_{jk}$ for $1\le j<k\le N$ are all independent Brownian motions (of variance $t/N$ on the main diagonal and $t/2N$ above it).  Then $U^N$ is the solution of the It\^o stochastic differential equation
\begin{equation} \label{e.SDE.UN} dU^N_t = iU^N_t\,dX^N_t - \frac12 U^N_t\,dt, \qquad U^N_0=I_N. \end{equation}

\subsection{Heat Kernel Computations and the Schur-Weyl duality\label{section Schur-Weyl Duality}}

There is no closed formula for the density of the heat kernel measure $\rho^N_t$ with $t>0$, even in the case $N=1$.  For the purposes of computing limits (in distribution, cf.\ Section \ref{section nc dist}), it is possible to compute many limiting moments due to a general feature of the Laplacian $\Delta_{\U_N}$: with the chosen scaling, it can be decomposed appropriately into a form $D+\frac{1}{N^2}L$ for $N$-independent operators $D$ and $L$ acting on an auxiliary space.  Given a decomposition of the auxiliary space into invariant finite-dimensional subspaces, the boundedness of the exponential map then gives
\[ e^{\frac{t}{2}\Delta_{\U_N}}\sim e^{\frac{t}{2}D}+O\left(\frac{1}{N^2}\right), \]
which reduces the computation of limiting moments to the (combinatorially tractable) computation of the flow of $D$.  In \cite{DHK2013,Kemp2013a,Kemp2013b} and \cite{Guillaume2013}, building on earlier work of \cite{LevyMaida2010,Rains1997,Sengupta} and others, the auxiliary space used was composed of so-called trace polynomials: polynomials in $U,\tr(U),\tr(U^2),\ldots$  In the complementary work of the second author \cite{Dahlqvist} and preceding papers of L\'evy \cite{Levy}, an alternative approach relying on the representation theory of $\U_N$ was used.  We outline this approach presently.

Fix a positive integer $n\in\N$.  Let $S_n$ denote the permutation group on $n$ letters.  For $1\le i\ne j\le n$, let $(i\,j)\in S_n$ denote the transposition.  For $\sigma\in S_n$ let $\#\sigma$ denote the number of cycles in the permutation $\sigma$.  Define two linear operators $L_n$ and $D_n$ on the (finite-dimensional) group algebra $\C[S_n]$ as follows: they are the linear extensions of
\begin{align} L_n(\sigma) &= \sum_{1\le i<j\le n\atop \#(\sigma\cdot(i\,j))>\#\sigma} \sigma\cdot(i\,j) \\
D_n(\sigma) &= \sum_{1\le i<j\le n\atop\#(\sigma\cdot(i\,j))<\#\sigma} \sigma\cdot(i\,j)
\end{align}
\noindent When the context is clear, we may drop the index $n$ (with the knowledge that $D=D_n$ and $L=L_n$ act only on the finite-dimensional space in which their argument lives).

Let $V$ be any vector space of dimension $\ge 2$.  The standard action of $S_n$ on $V^{\tensor n}$ is given by
\begin{equation} \label{e.Sn.action} \sigma\cdot (v_1\tensor\cdots\tensor v_n) = v_{\sigma^{-1}(1)}\tensor\cdots\tensor v_{\sigma^{-1}(n)}. \end{equation}
This yields a faithful representation, and so we will typically identify $\sigma$ with the corresponding endomorphism of $V^{\tensor n}$.  We will be concerned with this action when $V = \C^N$, so that (fixing the standard basis of $\C^N$) $\mathrm{End}(\C^N) = \M_N$.  We may then readily verify that, if $\sigma\in S_n$ has cycle decomposition $\sigma = c_1\cdots c_r$ with $c_k = (i^k_1 \cdots i^k_{\ell_k})$, then for $A_1,\ldots,A_n\in\M_N$,
\begin{equation} \label{e.trace.cycles} \Tr_{\M_N^{\tensor n}}(\sigma\cdot A_1\tensor\cdots\tensor A_n) = \Tr_{\M_N}(A_{i^1_1}\cdots A_{i^1_{\ell_1}})\cdots\Tr_{\M_N}(A_{i^r_1}\cdots A_{i^r_{\ell_r}}) \end{equation}
where we have emphasized over which space each trace is taken.  Note, in particular, that if $\sigma = (i_1\cdots i_n)$ is a cycle, then $\Tr(A_1\tensor\cdots\tensor A_n\cdot\sigma) = \Tr(A_{i_1}\cdots A_{i_n})$ is a single trace.

We define a function $F_n^N\colon\C[S_n]\times\M_N\to\C$, linear in the first argument, as follows:
\begin{equation} \label{e.F^N_n} F^N_n(\sigma,M) = \frac{1}{N^{\#\sigma}}\Tr(\sigma\cdot M^{\tensor n}). \end{equation}
The $N^{\#\sigma}$ in the denominator is a normalizing factor, since $\Tr(\sigma) = N^{\#\sigma}$ by \eqref{e.trace.cycles}; that is to say, \ $F^N_n(\sigma,M) = \Tr(\sigma\cdot  M^{\tensor n})/\Tr(\sigma)$, and so $F_n^N(\sigma,I_N) = 1$ for any $\sigma\in S_n$.  As usual, when the context is clear, we may drop the indices $N,n$ and refer to the function as $F(\sigma,M)$.  Note that any homogeneous degree $n$ polynomial in the entries of $M$ can be represented in the form $F(\sigma,M)$ for some element $\sigma\in\C[S_n]$.  This highlights the power of the following theorem.

\begin{theorem}[L\'evy, Schur-Weyl] \label{t.Schur-Weyl} Fix a positive integer $N$, a matrix $M\in\M_N$, and a time $t\ge 0$.  Then for any $n\in\N$ and $\sigma\in\C[S_n]$,
\[ \E[F_n^N(\sigma,U^N_tM)] = e^{-\frac{nt}{2}}F^N_n(e^{-t(L_n+\frac{1}{N^2}D_n)}\sigma,M). \]
\end{theorem}
\noindent (One might expect to need a $\frac{t}{2}$ in the exponential; in fact, this factor of $\frac12$ has been built into the operators $L_n$ and $D_n$.)

\subsection{Noncommutative distributions and convergence \label{section nc dist}}

Let $(\A,\t)$ be a $W^\ast$-probability space: a von Neumann algebra $\mathscr{A}$ equipped with a faithful, normal, tracial state $\t$.  Elements $a\in\A$ are referred to as (noncommutative) {\bf random variables}.  The {\bf noncommutative distribution} of any finite collection $a_1,\ldots,a_k\in A$ is the linear functional $\mu_{(a_1,\ldots,a_k)}$ on noncommutative polynomials defined by
\begin{equation} \label{e.nc.dist} \begin{aligned} \mu_{(a_1,\ldots,a_k)}\colon\C\langle X_1,\ldots,X_k\rangle&\to C \\
P&\mapsto \t(P(a_1,\ldots,a_k)).
\end{aligned} \end{equation}
Some authors explicitly include moments in $a_j,a_j^\ast$ in the definition of the distribution; we will instead refer to the $\ast$-distribution as the noncommutative distribution $\mu_{(a_1,a_1^\ast,\ldots,a_k,a_k^\ast)}$ explicitly when needed.  Note, when $a\in\A$ is normal, $\mu_{a,a^\ast}$ is determined by a unique probability measure $\mathrm{Law}_a$, the spectral measure of $a$, on $\C$ in the usual way:
\[ \int_\C f(z,\bar{z})\,\mathrm{Law}_a(dzd\bar{z}) = \mu_{a,a^\ast}(f), \quad f\in\C[X,X^\ast] \]
(i.e.\ when normal it suffices to restrict the noncommutative distribution to ordinary commuting polynomials).  In this case, the support $\supp\mathrm{Law}_a$ is equal to the spectrum $\mathrm{spec}(a)$.  If $u\in\mathscr{A}$ is unitary, $\mathrm{Law}_u$ is supported in the unit circle $\U_1$.  For example: a {\bf Haar unitary} is a unitary operator in $(\mathscr{A},\t)$ whose spectral measure is the uniform probability measure on $\U_1$ (equivalently $\t(u^n) = \delta_{n0}$ for $n\in\Z$).   In general, however, for a collection of elements $a_1,\ldots,a_k$ (normal or not) that do not commute, the noncommutative distribution is not determined by any measure on $\C$.

As a prominent example, let $A^N$ be a normal random matrix ensemble in $\M_N$: i.e.\ $A^N$ is a random variable defined on some probability space $(\Omega,\mathscr{F},\P)$, taking values in $\M_N$.  The distribution of $A^N$ as a random variable is a measure on $\M_N$; but for each instance $\omega\in\Omega$, the matrix $A^N(\omega)$ is a noncommutative random variable in the $W^\ast$-probability space $\M_N$, whose unique tracial state is $\tr$.  In this interpretation, the law $\mathrm{Law}_{A^N(\omega)}$ determined by its noncommutative distribution is precisely the empirical spectral distribution
\[ \mathrm{Law}_{A^N(\omega)} = \frac1N\sum_{j=1}^N \delta_{\lambda_j(A^N(\omega))}, \]
where $\lambda_1(A^N(\omega)),\ldots,\lambda_n(A^N(\omega))$ are the (random) eigenvalues of $A^N$.



Let $(A^N_1,\ldots,A^N_n)$ be a collection of random matrix ensembles, viewed as (random) noncommutative random variables in $(\M_N,\tr)$.  We will assume that the entries of $A^N$ are in $L^{\infty-}(\Omega,\mathscr{F},\P)$, meaning that they have finite moments of all orders.  The noncommutative distribution $\mu_{(A^N_1,\ldots,A^N_n)}$ is thus a random linear functional $\C\langle X_1,\ldots,X_n\rangle\to\C$; its value on a polynomial $P$ is the (classical) random variable $\tr(P(A^N_1,\ldots,A^N_n))$, cf.\ \eqref{e.nc.dist}.  Now, let $(\A,\t)$ be a $W^\ast$-probability space, and let $a_1,\ldots,a_n\in\A$.  Say that $(A^N_1,\ldots,A^N_n)$ {\bf converges in noncommutative distribution to $a_1,\ldots,a_n$ almost surely} if $\mu_{(A^N_1,\ldots,A^N_n)}\longrightarrow \mu_{(a_1,\ldots,a_n)}$ almost surely in the topology of pointwise convergence.  That is to say: convergence in noncommutative distribution means that all (random) mixed $\tr$ moments of the ensembles $A^N_j$ converge a.s.\ to the same mixed $\t$ moments of the $a_j$.
Later, a stronger notion of convergence emerged.

\begin{definition} \label{d.strong.conv} Let $\mx{A}^N = (A^N_1,\ldots,A^N_n)$ be random matrix ensembles in $(\M_N,\tr)$, and let $\mx{a}=(a_1,\ldots,a_n)$ be random variables in a $W^\ast$-probability space $(\A,\t)$.  Say that $\mx{A}^N$ {\bf converges strongly to} $\mx{a}$ if $\mx{A}^N$ converges to $\mx{a}$ almost surely in noncommutative distribution, and additionally
\[ \|P(A^N_1,\ldots,A^N_n)\|_{\M_N} \to \|P(a_1,\ldots,a_n)\|_\A \;\; a.s. \qquad \forall \; P\in\C\langle X_1,\ldots,X_n\rangle. \]
\end{definition}

This notion first appeared in the seminal paper \cite{Haagerup2005} of Haagerup and Thorbj{\o}rnsen, where they showed that if $X^N_1,\ldots,X^N_n$ are independent $\mathrm{GUE}^N$ random matrices, then they converge strongly to free semicircular random variables $(x_1,\ldots,x_n)$. The notion was formalized into Definition \ref{d.strong.conv} in the dissertation of Camille Male (cf.\ \cite{Male2012}), where the following generalization (an extension property of strong convergence) was proved.

\begin{theorem}[Male, 2012] \label{t.Male} Let $\mx{A}^N = (A^N_1,\ldots,A^N_n)$ be a collection of random matrix ensembles that converges strongly to some $\mx{a}=(a_1,\ldots,a_n)$ in a $W^\ast$-probability space $(\A,\t)$.  Let $\mx{X}^N = (X^N_1,\ldots,X^N_k)$ be independent Gaussian unitary ensembles independent from $\mx{A}^N$, and let $\mx{x} = (x_1,\ldots,x_k)$ be freely independent semicircular random variables in $\A$ all free from $\mx{a}$.  Then $(\mx{A}^N,\mx{X}^N)$ converges strongly to $(\mx{a},\mx{x})$.
\end{theorem}

\noindent (For a brief definition and discussion of free independence, see Section \ref{section free prob} below.) Later, together with the present first author in \cite{Collins2013}, Male proved a strong convergence result for Haar distributed random unitary matrices (which can be realized as $\lim_{t\to\infty} U^N_t$).

\begin{theorem}[Collins, Male, 2013] \label{t.Collins.Male} Let $\mx{A}^N = (A^N_1,\ldots,A^N_n)$ be a collection of random matrix ensembles that converges strongly to some $\mx{a}=(a_1,\ldots,a_n)$ in a $W^\ast$-probability space $(\A,\t)$.  Let $U^N$ be a Haar-distributed random unitary matrix independent from $\mx{A}^N$, and let $u$ be a Haar unitary operator in $\A$ freely independent from $\mx{a}$.  Then $(\mx{A}^N,U^N,(U^N)^\ast)$ converges strongly to $(\mx{a},u,u^\ast)$.
\end{theorem}

\noindent (The convergence in distribution in Theorem \ref{t.Collins.Male} is originally due to Voiculescu \cite{Voiculescu1998}; a simpler proof of this result was given in \cite{Collins-IMRN}.)  Our main Theorem \ref{t.main.2} can be viewed as combination and generalization of Theorems \ref{t.Male} and \ref{t.Collins.Male}.

Note that, for any matrix $A\in\M_N$ and any operator $a\in\A$,
\[ \|A\|_{\M_N} = \lim_{p\to\infty} \left(\tr\left[(AA^\ast)^{p/2}\right]\right)^{1/p}, \quad \text{and} \quad \|a\|_\A = \lim_{p\to\infty} \left(\tau\left[(aa^\ast)^{p/2}\right]\right)^{1/p}. \]
These hold because the states $\tr$ and $\t$ are faithful.  These are the noncommutative $L^p$-norms on $L^p(\M_N,\tr)$ and $L^p(\A,\t)$ respectively. The norm-convergence statement of strong convergence can thus be rephrased as an almost sure interchange of limits: if $\mx{A}^N$ converges a.s.\ to $\mx{a}$ in noncommutative distribution, then $\mx{A}^N$ converges to $\mx{a}$ strongly if and only if
\begin{equation} \label{e.strongconv2} \P\left(\lim_{N\to\infty}\lim_{p\to\infty} \|P(\mx{A}^N)\|_{L^p(\M_N,\tr)} = \lim_{p\to\infty} \|P(\mx{a})\|_{L^p(\A,\t)}\right)=1, \quad \forall\; P\in\C\langle X_1,\ldots,X_n\rangle. \end{equation}
\noindent If we fix $p$ instead of sending $p\to\infty$, the corresponding notion of ``$L^p$-strong convergence'' of the unitary Brownian motion $(U^N_t)_{t\ge 0}$ to the free unitary Brownian motion $(u_t)_{t\ge 0}$ was proved in the third author's paper \cite{Kemp2013a}.  This weaker notion of strong convergence does not have the same important applications as strong convergence, however.  As a demonstration of the power of true strong convergence, we give an application to the eigenvalues of the Jacobi process in Section \ref{section application}: the principal angles between subspaces randomly rotated by $U^N_t$ evolve a.s.\ with finite speed for all large $N$.

\subsection{Free probability, free stochastics, and free unitary Brownian motion \label{section free prob}}

We briefly recall basic definitions and constructions here, mostly for the sake of fixing notation.  The uninitiated reader is referred to the monographs \cite{VDN,NicaSpeicher2006}, and the introductions of the authors' previous papers \cite{CK2014,Kemp2013a,KNPS2012} for more details.

Let $(\mathscr{A},\tau)$ be a $W^\ast$-probability space.  Unital subalgebras $\A_1,\ldots,\A_m\subset\A$ are called {\bf free} or {\bf freely independent} if the following  property holds: given any  sequence of indices $k_1,\ldots,k_n\in\{1,\ldots,m\}$ that is consecutively-distinct (meaning $k_{j-1}\ne k_j$ for $1<j\le n$) and random variables $a_j\in \A_{k_j}$, if $\t(a_j)=0$ for $1\le j\le n$ then $\t(a_1\cdots a_n)=0$.  We say random variables $a_1,\ldots,a_m$ are freely independent if the unital $\ast$-subalgebras  $\A_j\equiv \langle a_j,a_j^\ast\rangle\subset\A$ they generate are freely independent.  Freeness is a moment factorization property: by centering random variables $a\to a-\t(a)1_\A$, freeness allows the (recursive) computation of any joint moment in free variables as a polynomial in the moments of the separate random variables.  In other words: the distribution $\mu_{(a_1,\ldots,a_k)}$ of a collection of free random variables is determined by the distributions $\mu_{a_1},\ldots,\mu_{a_k}$ separately.

A {\bf noncommutative stochastic process} is simply a one-parameter family $a=(a_t)_{t\ge 0}$ of random variables in some $W^\ast$-probability space $(\A,\t)$.  It defines a {\bf filtration}: an increasing (by inclusion) collection $\A_t$ of subalgebras of $\A$ defined by $\A_t \equiv W^\ast(a_s\colon 0\le s\le t)$, the von Neumann algebras generated by all the random variables $a_s$ for $s\le t$.  Given such a filtration $(\A_t)_{t\ge 0}$, we call a process $b=(b_t)_{t\ge 0}$ {\bf adapted} if $b_t\in\A_t$ for all $t\ge 0$.

A {\bf free additive Brownian motion} is a selfadjoint noncommutative stochastic process $x=(x_t)_{t\ge 0}$ in a $W^\ast$-probability space $(\A,\t)$ with the following properties:
\begin{itemize}
\item {\sc Continuity}: The map $\R_+\to\A\colon t\mapsto x_t$ is weak$^\ast$-continuous.
\item {\sc Free Increments}: For $0\le s\le t$, the additive increment $x_t-x_s$ is freely independent from $\A_s$ (the filtration generated by $x$ up to time $s$).
\item {\sc Stationary Increments}: For $0\le s\le t$, $\mu_{x_t-x_s} = \mu_{x_{t-s}}$.
\end{itemize}
It follows from the free central limit theorem that the increments must have the semicircular distribution: $\mathrm{Law}_{x_t} = \frac{1}{2\pi t}\sqrt{(4t-x^2)_+}\,dx$.  Voiculescu (cf.\ \cite{Voiculescu1991,Voiculescu1998,VDN} showed that free additive Brownian motions exist: they can be constructed in any $W^\ast$-probability space rich enough to contain an infinite sequence of freely independent semicircular random variables (where $x$ can be constructed in the usual way as an isonormal process).

In \cite{BianeSpeicher1998,BianeSpeicher2001} (and many subsequent works such as \cite{KNPS2012}), a theory of stochastic analysis built on $x$ was developed.  {\bf Free stochastic integrals} with respect to $x$ are defined precisely as in the classical setting: as $L^2(\A,\t)$-limits of integrals of simple processes, where for constant $a\in\A$, $\int_0^t \1_{[t_-,t_+]}(s)a\,dx_s$ is defined to be $a\cdot (x_{t_+}-x_{t_-})$.  Using the standard Picard iteration techniques, it is known that free stochastic integral equations of the form
\begin{equation} \label{e.fSDE0} a_t = a_0+\int_0^t \phi(s,a_s)\,ds + \int_0^t \sigma(s,a_s)\,dx_s \end{equation}
have unique adapted solutions for drift $\phi$ and diffusion $\sigma$ coefficient functions that are globally Lipschitz.  (Note: due to the noncommutativity, we should really be integrating a biprocess $\beta_t\# dx_t$ in the It\^o term, where $\beta_t\in\A_t\tensor\A_t$ so that it may act on both sides of $x$.  A one-sided process like the one in \eqref{e.fSDE0} is typically not self-adjoint, which limits $\phi,\sigma$ to be polynomials, and ergo linear polynomials due to the Lipschitz constraint.  That will suffice for our present purposes.)  Equations like \eqref{e.fSDE0} are often written in ``stochastic differential'' form as
\[ da_t = \phi(t,a_t)\,dt + \sigma(t,a_t)\,dx_t. \]
Given a free additive Brownian motion $x$, the associated {\bf free unitary Brownian motion} $u=(u_t)_{t\ge 0}$ is the solution to the free SDE
\begin{equation} \label{e.SDE.u} du_t = iu_t\,dx_t - \frac12 u_t\,dt, \qquad u_0=1. \end{equation}
This precisely mirrors the (classical) It\^o SDE \eqref{e.SDE.UN} that determined the Brownian motion $(U^N_t)_{t\ge 0}$ on $\U_N$.

The free unitary Brownian motion $(u_t)_{t\ge 0}$ was introduced in \cite{Biane1997a} (via the definition above).  In that paper, with more details in the subsequent \cite{Biane1997b}, together with independent statements of the same type in \cite{Rains1997}, the law $\mathrm{Law}_{u_t}$ was computed.  Since $u_t$ is unitary, this distribution is determined by a measure $\nu_t$ that is supported on the unit circle $\U_1$.  This measure is described as follows.

\begin{theorem}[Biane 1997] \label{t.nu.t} For $t>0$, $\nu_t$ has a continuous density $\varrho_t$ with respect to the normalized Haar measure on $\U_1$.  For $0<t<4$, its support is the connected arc
\begin{equation} \label{e.supp.nu_t} \supp\nu_t = \left\{e^{i\theta}\colon |\theta|\le \frac12\sqrt{t(4-t)}+\arccos\left(1-\frac{t}{2}\right)\right\}, \end{equation}
while $\supp\nu_t=\U_1$ for $t\ge 4$.  The density $\varrho_t$ is real analytic on the interior of the arc.  It is symmetric about $1$, and is determined by
$\varrho_t(e^{i\theta}) = \Re \kappa_t(e^{i\theta})$ where $z=\kappa_t(e^{i\theta})$ is the unique solution (with positive real part) to
\[ \frac{z-1}{z+1}e^{\frac{t}{2}z} = e^{i\theta}. \]
\end{theorem}
\noindent Note that the arc \eqref{e.supp.nu_t} is the spectrum $\mathrm{spec}(u_t)$ for $0<t<4$; for $t\ge 4$, $\mathrm{spec}(u_t) = \U_1$.

With this description, one can also give a characterization of the free unitary Brownian motion similar to the invariant characterization of the Brownian motion $(U^N_t)_{t\ge 0}$ on page \pageref{U^N invariance}.  That is, $(u_t)_{t\ge 0}$ is the unique unitary-valued process that satisfies:
\begin{itemize}
\item {\sc Continuity}: The map $\R_+\to\A\colon t\mapsto u_t$ is weak$^\ast$ continuous.
\item {\sc Freely Independent Multiplicative Increments}: For $0\le s\le t$, the multiplicative increment $u_s^{-1}u_t$ is independent from the filtration up to time $s$ (i.e.\ from the von Neumann algebra $\A_s$ generated by $\{u_r\colon 0\le r\le s\}$).
\item {\sc Stationary Increments with Distribution $\nu$}: For $0\le s\le t$, the multiplicative increment $u_s^{-1}u_t$ has distribution given by the law $\nu_{t-s}$.
\end{itemize}

\newpage

\section{The Hard Edge of the Spectrum \label{section hard edge}}

This section is devoted to the proof of our ``hard edge'' theorem for the spectrum of a single time marginal $U^N_t$.  We begin by showing how Theorem \ref{t.main.1} follows from Theorem \ref{t.moment.rate.conv}, and recast the conclusion as a strong convergence statement in Corollary \ref{c.strong.conv}.  Section \ref{section hard edge 2} is then devoted to the proof of the moment growth bound of Theorem \ref{t.moment.rate.conv}.

\subsection{Strong convergence and the proof of Theorem \ref{t.main.1}\label{section hard edge 1}}

We begin by briefly recalling some basic Fourier analysis on the circle $\U_1$.  For $f\in L^2(\U_1)$, its Fourier expansion is
\[ f(w) = \sum_{n\in\Z} \hat{f}(n)w^n, \quad \text{where} \quad \hat{f}(n) = \int_{\U_1} f(w)w^{-1}\,dw, \]
where $dw$ is the normalized Lebesgue measure on $\U_1$.   For $p>0$, the {\bf Sobolev space} $H_p(\U_1)$ is defined to be
\begin{equation} \label{e.d.Hp} H_p(\U_1) = \left\{f\in L^2(\U_1)\colon \|f\|_{H_p}^2\equiv \sum_{n\in\Z} (1+n^2)^p|\hat{f}(n)|^2 < \infty\right\}. \end{equation}
\noindent If $\ell>k\ge1$ are integers, and $\ell\ge p \ge k+\frac12$, then $C^\ell(\U_1)\subset H_p(\U)\subset C^k(\U_1)$; it follows that $H_\infty(\U_1)\equiv\bigcap_{p\ge 0} H_p(\U_1) = C^\infty(\U_1)$. These are standard Sobolev imbedding theorems (that hold for smooth manifolds); for reference, see \cite[Chapter 5.6]{Evans2010} and \cite[Chapter 3.2]{Saloff-Coste2002}.

\bigskip

Theorem \ref{t.moment.rate.conv} yields the following estimate on moment growth tested against Sobolev functions disjoint from the limit support.

\begin{proposition} \label{p.edge.1} Fix $0\le t<4$.  Let $f\in H_5(\U_1)$ have support disjoint from $\supp\nu_t$.  There is a constant $C(f)>0$ such that, for all $N\in\N$,
\begin{equation} \label{e.moment.growth} \left|\E\tr[f(U^N_t)]\right| \le \frac{t^2C(f)}{N^2}. \end{equation}
\end{proposition}

\begin{proof} Denote by $\nu^N_t(n) \equiv \E\tr[(U^N_t)^n]$ and $\nu_t(n) \equiv \int_{\U_1} w^n\,\nu_t(dw) = \lim_{N\to\infty} \nu_t^N(n)$.  Expanding $f$ as a Fourier series, we have
\begin{equation} \label{e.p3.1.1} \E\tr[f(U^N_t)] = \sum_{n\in\Z} \hat{f}(n)\E\tr[(U^N_t)^n] = \sum_{n\in\Z} \hat{f}(n) \nu^N_t(n). \end{equation}
By the assumption that $\supp f$ is disjoint from $\supp\nu_t$, we have
\begin{equation} \label{e.p3.1.2} 0 = \int_{\U_1} f\,d\nu_t = \sum_{n\in\Z} \hat{f}(n)\int_{\U_1} w^n\,\nu_t(dw) = \sum_{n\in\Z} \hat{f}(n) \nu_t(n). \end{equation}
Combining \eqref{e.p3.1.1} and \eqref{e.p3.1.2} with Theorem \ref{t.moment.rate.conv} yields
\[ \left|\E\tr[f(U^N_t)]\right| \le \sum_{n\in\Z} |\hat{f}(n)||\nu^N_t(n)-\nu_t(n)| \le \sum_{n\in\Z} |\hat{f}(n)|\cdot \frac{t^2n^4}{N^{2}}. \]
By assumption $f\in H_5(\U_1)$, and so
\[ \sum_{n\in\Z} n^4|\hat{f}(n)| = \sum_{n\in\Z\setminus\{0\}} \frac{1}{n}\cdot n^5|\hat{f}(n)| \le \left(\sum_{n\in\Z\setminus\{0\}} \frac{1}{n^2}\right)^{1/2}\left(\sum_{n\in\Z} n^{10}|\hat{f}(n)|^2\right)^{1/2} \le \frac{\pi}{\sqrt{3}}\|f\|_{H_5}<\infty. \]

Taking $C(f) = \frac{\pi}{\sqrt{3}}\|f\|_{H_5}$ concludes the proof.
\end{proof}

We now use Proposition \ref{p.edge.1} to give an improved variance estimate related to \cite[Propositions 6.1, 6.2]{LevyMaida2010}.

\begin{proposition} \label{p.edge.2} Fix $0\le t<4$.  Let $f\in C^6(\U_1)$ with support disjoint from $\supp\nu_t$.  There is a constant $C'(f)>0$ such that, for all $N\in\N$,
\[ \Var[\Tr(f(U^N_t))] \le \frac{t^3C'(f)}{N^2}. \]
\end{proposition}

\begin{proof} In the proof of \cite[Proposition 3.1]{LevyMaida2010} (on p.\ 3179), and also in \cite[Proposition 4.2 \& Corollary 4.5]{CebronKemp2014}, the desired variance is shown to have the form
\begin{equation} \label{e.p3.3.0} \Var[\Tr(f(U^N_t))] = \int_0^t \E\tr[f'(U^N_sV^N_{t-s})f'(U^N_sW^N_{t-s})]\,ds  \end{equation}
where $U^N,V^N,W^N$ are three independent Brownian motions on $\U_N$.  For fixed $s\in[0,t]$, we apply the Cauchy-Schwarz inequality twice and use the equidistribution of $U_s^NV_{t-s}^N$ and $U_s^NW_{t-s}^N$  to yield
\[ \left|\E\tr[f'(U^N_sV^N_{t-s})f'(U^N_sW^N_{t-s})]\right| \le \E\left[\left|\tr[f'(U^N_sV^N_{t-s})^2]\right|^{1/2}\cdot \left|\tr[f'(U^N_sW^N_{t-s})^2]\right|^{1/2}\right] \le \E\tr[f'(U_s^NV_{t-s}^N)^2]. \]
Since $U^N$ and $V^N$ are independent, $(U^N_s,V^N_{t-s})$ has the same distribution as $(U^N_s,(U^N_s)^{-1}U^N_t)$ (as the increments are independent and stationary). Thus $\E\tr[f'(U_s^NV_{t-s}^N)^2] = \E\tr[f'(U^N_t)^2]$, and so, integrating, we find
\begin{equation} \label{e.p3.3.1} \Var[\Tr(f(U^N_t))] \le t\E\tr[f'(U^N_t)^2]. \end{equation}
Since $f\in C^6(\U_1)$, the function $(f')^2$ is $C^5\subset H_5$, and the result now follows from Proposition \ref{p.edge.1}, with $C'(f) = C((f')^2)$.
\end{proof}

This brings us to the proof of the ``hard edge'' theorem.

\begin{proof}[Proof of Theorem \ref{t.main.1} assuming Theorem \ref{t.moment.rate.conv}] Fix a closed arc $\alpha \subset \U_1$ that is disjoint from $\supp\nu_t$.  Let $f$ be a $C^\infty$ bump function with values in $[0,1]$ such that $\left.f\right|_\alpha = 1$ and $\supp f\cap \supp\nu_t=\emptyset$.  Then
\begin{equation} \label{e.proof.t1.1} \P(\mathrm{spec}(U^N_t)\cap \alpha\ne\emptyset) \le \P(\Tr[f(U^N_t)]\ge 1). \end{equation}
We now apply Chebyshev's inequality, in the following form: let $Y=\Tr[f(U^N_t)]$.  Then, assuming $1-\E(Y)>0$, we have
\[ \P(Y\ge 1) = \P(Y-\E(Y)\ge 1-\E(Y)) \le \frac{\Var(Y)}{(1-\E(Y))^2}. \]
In our case, we have $|\E(Y)| = |\E\Tr[f(U^N_t)]| = N|\E\tr[f(U^N_t)]| \le \frac{t^2C(f)}{N}$ by Proposition \ref{p.edge.1}.  Thus, there is $N_0$ (depending only on $f$ and $t$) so that $(1-\E\Tr[f(U^N_t)])^2\ge\frac12$ for $N\ge N_0$.  Combining this with \eqref{e.proof.t1.1} yields
\[ \P(\mathrm{spec}(U^N_t)\cap \alpha\ne\emptyset) \le 2\Var[\Tr(f(U^N_t))] \quad \text{for} \quad N\ge N_0. \]
Now invoking Proposition \ref{p.edge.2}, we find that this is $\le \frac{2t^3 C'(f)}{N^2}$ whenever $N\ge N_0$. It thus follows from the Borel-Cantelli lemma that $ \P(\mathrm{spec}(U^N_t)\cap \alpha\ne\emptyset)=0$ for all but finitely many $N$.

Thus, we have shown that, for any closed arc $\alpha$ disjoint from $\supp\nu_t$, with probability $1$, $\mathrm{spec}(U^N_t)$ is contained in $\U_1\setminus\alpha$ for all large $N$.  In particular, fixing any open arc $\beta\subset\U_1$ containing $\supp\nu_t$, this applies to $\alpha = \U_1\setminus\beta$.  I.e.\ $\mathrm{spec}(U^N_t)$ is a.s.\ contained in any neighborhood of $\supp\nu_t$ for all sufficiently large $N$.  This suffices to prove the theorem: because $\mathrm{Law}_{U^N_t}$ converges weakly almost surely to the measure $\nu_t$ which possesses a strictly positive continuous density on its support, any neighborhood of the spectrum of $U^N_t$ eventually covers $\supp\nu_t$.  \end{proof}

Thus, we have proved Theorem \ref{t.main.1} under the assumption that Theorem \ref{t.moment.rate.conv} is true.  Before turning to the proof of this latter result, let us recast Theorem \ref{t.main.1} in the language of strong convergence, as we will proceed to generalize this to the fully noncommutative setting in Section \ref{section strong convergence}.

\begin{corollary} \label{c.strong.conv} For $N\in\N$, let $(U^N)_{t\ge 0}$ be a Brownian motion on $\U_N$.  Let $(u_t)_{t\ge 0}$ be a free unitary Brownian motion. Then for any fixed $t\ge 0$, $(U^N_t,(U^N_t)^\ast)$ converges strongly to $(u_t,u_t^\ast)$. \end{corollary}

\begin{proof} Since $U^N_t\to u_t$ in noncommutative distribution, strong convergence is the statement that
\[ \|P(U_t^N,(U_t^N)^\ast)\|\to \|P(u_t,u_t^\ast)\| \]
in operator norms. Fix a noncommutative polynomial $P$ in two variables, and let $p$ be the unique Laurent polynomial in one variable so that $P(U,U^\ast) = p(U)$ for every unitary operator $U$.  
Since $U_t^N$ is normal, $\|p(U_t^N)\| = \max\{|\lambda|\colon \lambda\in p(\mathrm{spec}(U^N_t))\}$; similarly, $\|p(u_t)\| = \max\{|\lambda|\colon \lambda\in p(\supp \nu_t)\}$ where $\supp\nu_t$ is the arc in \eqref{e.supp.nu_t}.

Let $\Lambda_p^N = |p|(\mathrm{spec}(U^N_t))$, and let $\Lambda_p = |p|(\supp\nu_t)$.  Since $\mathrm{spec}(U^N_t)$ converges to $\supp\nu_t$ in Hausdorff distance and all the sets are compact, it follows easily from the continuity of $|p|$ (on the unit circle) that $\Lambda_p^N$ converges to $\Lambda_p$ in Hausdorff distance as well.  Now, suppose for a contradiction that $\max \Lambda_p^N$ does not converge to $\max\Lambda_p$.  Note that $\Lambda_p^N$ is compact so contains its maximum, and since $\Lambda_p^N$ converges to the compact set $\Lambda_p$ in Hausdorff distance, it follows that the sequence $(\max\Lambda_p^N)_{N=1}^\infty$ is bounded, hence contains a convergent subsequence $(\max\Lambda^{N_k}_p)_{k=1}^\infty$.  Denote the limit of this sequence by $m$.  Now, $\Lambda^{N_k}_p$ converges to $\Lambda_p$ in Hausdorff distance; in particular, for any fixed $\e>0$, $\Lambda^{N_k}_p\subseteq (\Lambda_p)_\e$ (the set of all points distance $\le\e$ away from $\Lambda_p$).  Hence $\max\Lambda^{N_k}_p \in \Lambda^{N_k}_p <\max\Lambda_p+\e$, and so the limit $m\le \max\Lambda_p+\e$.  This holds for each $\e>0$, and so $m\le\max\Lambda_p$.  By assumption $m\ne\max\Lambda_p$, and so $m<\max\Lambda_p$.

Thus $\lim_{k\to\infty}\max\Lambda^{N_k}_p < \max\Lambda_p$.  In particular, there is some $\d>0$ so that, for all large $k$, $\max\Lambda^{N_k}_p < \max\Lambda_p-\delta$.  But the fact that $\Lambda^{N_k}_p\to \Lambda_p$ in Hausdorff distance implies that $\Lambda_p\subseteq (\max\Lambda^{N_k}_p)_\d$ for all large $k$.  This is a contradiction.  So we have shown that $\|p(U_t^N)\| = \max\Lambda_p^N \to \max\Lambda_p = \|p(u_t)\|$, as desired.  \end{proof}

\begin{remark} \label{r.strong.conv} In fact, the converse of Corollary \ref{c.strong.conv} also holds: strong convergence of $U^N_t\to u_t$ (for a fixed $t<4$) implies convergence of the spectrum in Hausdorff distance.  Indeed, suppose we know strong convergence.  Since all the operators involved are unitaries, we may extend the test function space to continuous functions on the unit circle $\U_1$: let $f\in C(\U_1)$, fix $\e>0$, and by the Weierstrass approximation theorem, choose a polynomial with $\|p-f\|_{L^\infty(\U_1)}<\frac{\e}{4}$.  Applying unitary functional calculus, this means $\|p(U)-f(U)\|<\frac{\e}{4}$ for any unitary operator $U$.  By assumption of strong convergence, we know $|\|p(U^N_t)\|- \|p(u_t)\||<\frac{\e}{2}$ for all large $N$, and it therefore follows that $|\|f(U^N_t)\|-\|f(u_t)\||<\e$ for large $N$.  So $\|f(U^N_t)\|\to \|f(u_t)\|$.

Now, let $\alpha\subset\U_1$ be a closed arc disjoint from $\supp\nu_t$, and let $f$ be a continuous bump function with values in $[0,1]$ such that $\left.f\right|_\alpha=1$ and $\supp f \cap \supp\nu_t=\emptyset$.  By the strong convergence assumption, $\|f(U^N_t)\|\to\|f(u_t)\|$; but $\supp f$ is disjoint from the spectrum of $u_t$, so $\|f(u_t)\|=0$.  Hence, $f(U_t^N)\to 0$ in norm, which shows that $U_t^N$ asymptotically has no eigenvalues in $\alpha$.  As above, this shows that $\mathrm{spec}(U^N_t)$ is eventually contained in any neighborhood of $\supp\nu_t$; the other half of the convergence in Hausdorff distance follows from the convergence in distribution (and strict positivity of the limit density $\nu_t$ on $\supp\nu_t$).

When $t\ge 4$, $\supp\nu_t =\U_1$, and strong convergence becomes vacuously equivalent to the known convergence in distribution.
\end{remark}

\subsection{The proof of Theorem \ref{t.moment.rate.conv}\label{section hard edge 2}}

We will actually prove the following Cauchy sequence growth estimate.  We again use the notation $\nu^N_t(n) = \E[\tr(U^N_t)^n]$.

\begin{proposition} \label{p.moment.rate.Cauchy} Let $N,n\in\N$, and fix $t\ge 0$.  Then
\begin{equation} \label{e.moment.rate.conv.Cauchy} \left|\nu^N_t(n)-\nu^{2N}_t(n)\right| \le \frac{3t^2n^4}{4N^2}. \end{equation}
\end{proposition}
\noindent This is the main technical result of the first part of the paper, and its proof will occupy most of this section.  Let us first show how Theorem \ref{t.moment.rate.conv} follows from Proposition \ref{p.moment.rate.Cauchy}.

\begin{proof}[Proof of Theorem \ref{t.moment.rate.conv} assuming Proposition \ref{p.moment.rate.Cauchy}] Since $\lim_{N\to\infty} \nu^N_t(n) = \nu_t(n)$, we have the following convergent telescoping series:
\[ |\nu^N_t(n)-\nu_t(n)| = \left|\sum_{k=0}^\infty \left(\nu^{N2^k}_t(n)-\nu^{N2^{k+1}}_t(n)\right)\right| \le \sum_{k=0}^\infty \left|\nu^{N2^k}_t(n)-\nu^{N2^{k+1}}_t(n)\right|. \]
Now apply \eqref{e.moment.rate.conv.Cauchy} with $N$ replaced by $N2^k$, we find
\[ \left|\nu^{N2^k}_t(n)-\nu^{N2^{k+1}}_t(n)\right| \le \frac34\frac{t^2n^4}{(N2^k)^2} = \frac34\frac{1}{4^k}\frac{t^2n^4}{N^2}. \]
Summing the geometric series proves the theorem.  \end{proof}

\begin{remark} The bound \eqref{e.moment.rate.conv} on the speed of convergence $\nu^N_t(n)\to \nu_t(n)$ depends polynomially on $n$; this is crucial to the proof.  In \cite[Section 3.3]{Kemp2013a}, the author proved a bound of the form $K(t,n)/N^2$, where $K(t,n) \sim \frac{tn^2}{2}\exp(\frac{tn^2}{2})$.  This growth in $n$ is much too large to get control over test functions $f$ that are only in a Sobolev space, or even in $C^\infty(\U_1)$; the largest class of functions for which this Fourier series is summable is an ultra-analytic Gevrey class.  That blunter estimate was proved not only for $U^N$, however, but for a family of diffusions on $\GL_N$ including both $U^N$ and the Brownian motion on $\GL_N$.  It remains open whether a polynomial bound like \eqref{e.moment.rate.conv} holds for this wider class of diffusions.
\end{remark}

Hence, we turn to the proof of Proposition \ref{p.moment.rate.Cauchy}. Fix a Brownian motion $U^{2N}$ on $\U_{2N}$, along with two Brownian motions $U^{N,1},U^{N,2}$ on $\U_N$, so that the processes $U^{2N},U^{N,1},U^{N,2}$ are all independent.  For $t\ge 0$, let $B^{2N}_t\in\U_{2N}$ denote the block diagonal random matrix
\[ B^{2N}_t = \left[\begin{array}{cc} U^{N,1}_t & 0 \\ 0 & U^{N,2}_t \end{array}\right] \in \U_{2N}. \]
Let us introduce the notation 
\[ A^{2N}_s = U^{2N}_{t-s}B^{2N}_s \]
as this process will be used very often in what follows.

Now, using the notation of \eqref{e.F^N_n}, for any $n\in\N$ and any element $\sigma\in\C[S_n]$, we define
\begin{equation}  F(s,\sigma) = \E[F^{2N}_n(\sigma,A^{2N}_s)], \end{equation}
where for readability we hide the explicit dependence of $F(s,\sigma)$ on $N,n,t$.  Taking $s=0$, since $B^{2N}_0 = I_{2N}$, this gives $\E[F^{2N}_n(\sigma,U_t^{2N})]$, while taking $s=t$ yields $\E[F^{2N}_n(\sigma,B^{2N}_t)]$.  Now taking $\sigma=(1\,\cdots\,n)$ to be the full cycle (so that $N^{\#\sigma}=N$), from the definition of normalized trace we have
\[ F(0,(1\,\cdots\,n)) = \nu^{2N}_t(n), \qquad F(t,(1\,\cdots\,n)) = \nu^N_t(n). \]
Thus, the quantity we wish to estimate may be computed as
\begin{equation} \label{e.del.s.of.F} \nu^N_t(n) - \nu^{2N}_t(n) = \int_0^t \frac{d}{d s} F(s,(1\,\cdots\,n))\,ds \end{equation}
provided this derivative exists.  We now proceed to show that it does, and compute it.

Denote by $P_1,P_2\in\M_{2N}$ the projection matrices from $\C^{2N}=\C^N\oplus\C^N$ onto the two factors; to be precise, $P_1=\mathrm{diag}[1,\ldots,1,0,\ldots,0]$ and $P_2 = I_{2N}-P_1$.  For any $A,B\in\M_{2N}$ and $1\le i<j\le n$, denote by $(A\tensor B)_{i,j}$ the matrix
\begin{equation} \label{e.notation.i,j} (A\tensor B)_{i,j} = I^{\tensor i-1} \tensor A \tensor I^{\tensor j-i-1} \tensor B \tensor I^{\tensor n-j} \end{equation}
where $I=I_{2N}$ is the identity matrix in $\M_{2N}$. 

\begin{lemma} \label{l.G'(s)} Fix $t>0$ and $N,n\in\N$, and let $G\colon[0,t]\to\M_{2N}^{\tensor n}$ denote the function $G(s) = \E[(A^{2N}_s)^{\tensor n}]$.  Then $G\in C^1[0,t]$, and
\begin{equation} \label{e.G'} G'(s) = \frac{1}{2N} G(s)\sum_{1\le i<j\le n}(i\,j)\left[I-2(P_1\tensor P_1)_{i,j}-2(P_2\tensor P_2)_{i,j}\right] \end{equation}
where $(i\,j)$ denotes the Schur-Weyl representation of the transposition in $S_n$, cf.\  \eqref{e.Sn.action}. \end{lemma}

\begin{proof} To begin, note that $U^{2N}$ and $B^{2N}$ are independent, and so it follows that
\begin{equation} \label{e.G.product.indep} G(s) = \E[(U^{2N}_{t-s})^{\tensor n}]\E[(B^{2N}_s)^{\tensor n}] \equiv G_1(s)G_2(s), \end{equation}
where both factors $G_1,G_2$ are continuous (since they are expectations of polynomials in diffusions).  Using the SDE \eqref{e.SDE.UN} and  applying It\^o's formula to the diffusion $B^{2N}$ shows that there is an $L^2$-martingale $(M^{2N}_s)_{s\ge 0}$ such that
\[ d\left((B^{2N}_s)^{\tensor n}\right) = dM_s - \frac{n}{2} (B^{2N}_s)^{\tensor n}\,dt -\frac1N \sum_{\substack{1\le i<j\le n\\ 1\le a,b\le N}}\sum_{\ell=1}^2 (B^{2N}_s)^{\tensor n}\cdot (E_{a+\ell N,b+\ell N}\tensor E_{b+\ell N,a+\ell N})_{i,j}\,ds \]
where $E_{c,d}\in\M_{2N}$ is the standard matrix unit (all $0$ entries except a $1$ in entry $(c,d)$) with indices written modulo $2N$.  This simplifies as follows: recalling that we identify an element $\sigma\in\C[S_n]$ with a matrix (in this case in $\M_{2N}$) via the faithful action \eqref{e.Sn.action}, we can write this SDE in the form
\begin{equation} \label{e.SDE.tensor.n} d\left((B^{2N}_s)^{\tensor n}\right) = dM_s - \frac{n}{2} (B^{2N}_s)^{\tensor n}\,ds -\frac1N \sum_{1\le i<j\le n} \sum_{\ell=1}^2 (B^{2N}_s)^{\tensor n}(i\,j)(P_\ell\tensor P_\ell)_{i,j}\,ds. \end{equation}
It follows that $G_2\in C^1[0,t]$, and
\begin{equation} \label{e.G2'} G_2'(s) = -\frac{n}{2}G_2(s)-\frac1N \sum_{1\le i<j\le n} \sum_{\ell=1}^2 \E[(B^{2N}_s)^{\tensor n}(i\,j)(P_\ell\tensor P_\ell)_{i,j}]. \end{equation}
At the same time, a similar calculation with It\^o's formula applied with \eqref{e.SDE.UN} shows that there is an $L^2$-martingale $(\widetilde{M}^{2N}_s)_{s\ge 0}$ such that
\begin{equation} \label{e.SDE.tensor.n.2} d((U^{2N}_s)^{\tensor n}) = d\widetilde{M}_s - \frac{n}{2}(U^{2N}_s)^{\tensor n}\,ds - \frac{1}{2N}\sum_{1\le i<j\le n} (U^{2N}_s)^{\tensor n}(i\,j)\,ds \end{equation}
which, changing $s\mapsto t-s$, implies that $G_1$ is $C^1[0,t]$ and
\begin{equation} \label{e.G1'} G_1'(s) = \frac{n}{2}G_1(s) + \frac{1}{2N}\sum_{1\le i<j\le n} \E[(U^N_{t-s})^{\tensor n}(i\,j)]. \end{equation}
Combining \eqref{e.G1'} and \eqref{e.G2'}, the product rule $G'(s) = G_1(s)G_2'(s) + G_2(s)G_1'(s)$ shows that $G\in C^1[0,t]$.  Using $G=G_1\cdot G_2$ again when recombining, we see that the $\frac{n}{2}$ terms cancel; moreover, the same recombination due to independence yields
\[ G'(s) = \frac{1}{2N}\sum_{1\le i<j\le n} \E[(U^N_{t-s})^{\tensor n}(i\,j) (B^N_s)^{\tensor n}] - \frac1N\sum_{1\le i<j\le n}\sum_{\ell=1}^2 \E[(A^{2N}_s)^{\tensor n}(i\,j)(P_\ell\tensor P_\ell)_{i,j}]. \]
Finally, in the first term, notice that $(i\,j)(B^N_s)^{\tensor n}=(B^N_s)^{\tensor n}(i\,j)$ (since the Schur-Weyl representation of any permutation commutes with any matrix of the form $B^{\tensor n}$). Hence, we have
\[ \E[(U^N_{t-s})^{\tensor n}(i\,j) (B^N_s)^{\tensor n}] = \E[(U^N_{t-s})^{\tensor n}]\E[(i\,j)(B^N_s)^{\tensor n}] = \E[(U^N_{t-s})^{\tensor n}]\E[(B^N_s)^{\tensor n}(i\,j)] = G(s)(i\,j). \]
Similarly factoring out the $G(s)$ from the second term yields the result.
\end{proof}

Now, note that
\begin{equation} \label{e.F(s,sigma).1} F(s,\sigma) = \E[F^{2N}_n(\sigma,A^{2N}_s)] = \frac{1}{(2N)^{\#\sigma}}\E\Tr[\sigma\cdot(A^{2N}_s)^{\tensor n}] = \frac{1}{(2N)^{\#\sigma}}\Tr[\sigma\cdot G(s)]. \end{equation}
This shows that $F(\,\cdot\,,\sigma)\in C^1[0,t]$, and so \eqref{e.del.s.of.F} is valid.  Further, we can use \eqref{e.G'} to compute the integrand there.  To that end, we introduce the following auxiliary functions: for $p,q\in\N$ and $s\in[0,t]$, denote
\begin{equation} \begin{aligned} \label{e.Hpq} H_{p,q}(s) &= \frac{1}{4N^2}\E[\Tr((A^{2N}_s)^p)\Tr((A^{2N}_s)^q)] \\
&\hspace{1in} -\frac{1}{2N^2}\sum_{\ell=1}^2 \E[\Tr((A^{2N}_s)^pP_\ell)\Tr((A^{2N}_s)^qP_\ell)]. \end{aligned} \end{equation}

\begin{lemma} For $0\le s\le t$,
\begin{equation} \label{e.F(s,sigma).2} \frac{d}{ds}F(s,(1\,\cdots\,n)) = \frac{n}{2}\sum_{p=1}^{n-1} H_{p,n-p}(s). \end{equation}
\end{lemma}

\begin{proof}  Applying \eqref{e.F(s,sigma).1} with $\sigma=(1\,\cdots\,n)$, and using \eqref{e.G'}, we have
\begin{align*} \frac{d}{ds}F(s,(1\,\cdots\,n)) &= \frac1{2N}\Tr[(1\,\cdots\,n)G'(s)] \\
&= \frac{1}{4N^2}\sum_{1\le i<j\le n}\left[\Tr[(1\,\cdots\,n)G(s)(i\,j)] - 2\sum_{\ell=1}^2 \Tr[(1\,\cdots\,n)G(s)(i\,j)(P_\ell\tensor P_\ell)_{i,j}\right]. \end{align*}
Using the trace property and noting that $(i\,j)(1\,\cdots\,n) = (1\,\cdots\,i-1\,j\,\cdots\,n)(i\,\cdots\,j-1)$, a simple calculation shows that
\[ \Tr[(1\,\cdots\,i-1\,j\,\cdots\,n)(i\,\cdots\,j-1)G(s)] = \E[\Tr((A^{2N}_s)^{j-i})\Tr((A^{2N}_s)^{n-(j-i)})]. \]
A similar calculation shows that
\[ \Tr[(1\,\cdots\,n)G(s)(i\,j)(P_\ell\tensor P_\ell)_{i,j}] = \E[\Tr((A^{2N}_s)^{j-i}P_\ell)\Tr((A^{2N}_s)^{n-(j-i)}P_\ell)]. \]
Thus, we have
\begin{align*} &\frac{d}{ds}F(s,(1\,\cdots\,n))\\
& = \frac{1}{4N^2}\sum_{1\le i<j\le n} \left[ \E[\Tr((A^{2N}_s)^{j-i})\Tr((A^{2N}_s)^{n-(j-i)})] - 2\E[\Tr((A^{2N}_s)^{j-i}P_\ell)\Tr((A^{2N}_s)^{n-(j-i)}P_\ell)]\right]. \end{align*}
Breaking this into two sums, each has the form $\sum_{1\le i<j\le n} h_{j-i,n-(j-i)}$ for some symmetric function $h\colon\{1,\ldots,n-1\}^2\to\C$.  For such a sum in general we have
\[ S\equiv \sum_{1\le i<j\le n} h_{j-i,n-(j-i)} = \sum_{p=1}^{n-1} \sum_{\substack{1\le i<j\le n\\ j-i=p}} h_{p,n-p} = \sum_{p=1}^{n-1} (n-p) h_{p,n-p} \]
since the number of $(i,j)$ with $1\le i<j\le n$ and $j-i=p$ is $(n-p)$.  Now using the symmetry and reindexing by $q=n-p$ we have
\[ 2S = \sum_{p=1}^{n-1} (n-p)h_{p,n-p} + \sum_{q=1}^{n-1} q h_{n-q,q} = \sum_{p=1}^{n-1} (n-p)h_{p,n-p} + \sum_{p=1}^{n-1} ph_{p,n-p} = n\sum_{p=1}^{n-1} h_{p,n-p}. \]
Applying this with the above summations yields the result.
\end{proof}

From \eqref{e.F(s,sigma).2} and \eqref{e.del.s.of.F}, we therefore have
\begin{equation} \label{e.nu.intermsof.Hpq} \nu_t^N(n)-\nu_t^{2N}(n) = \frac{n}{2}\sum_{p=1}^{n-1} \int_0^t H_{p,n-p}(s)\,ds. \end{equation}
It now behooves us to estimate the terms $H_{p,n-p}$, cf.\ \eqref{e.Hpq}.  Since the result (Proposition \ref{p.moment.rate.Cauchy}) gives an $O(1/N^2)$-estimate, we must show that $H_{p,q}(s) = O(1/N^2)$.  Note, however, that \eqref{e.Hpq} involves expectations of unnormalized traces of powers of $A^{2N}_s$.  As this process possesses a limit noncommutative distribution in terms of normalized traces, the leading order contribution of the first term in $H_{p,q}$ is $O(1)$.  In fact, there are cancellations between the two terms: $H_{p,q}(s)$ is actually a difference of covariances.

\begin{lemma} \label{l.Hpq.cov} For $s\ge 0$ and $p,q\in\N$,
\begin{equation} \label{e.Hpq.cov} N^2 H_{p,q}(s) = \frac{1}{4}\mathrm{Cov}[\Tr((A^{2N}_s)^p),\Tr((A^{2N}_s)^q)] -\mathrm{Cov}[\Tr((A^{2N}_s)^pP_1),\Tr((A^{2N}_s)^qP_1)]. \end{equation}
\end{lemma}

\begin{proof} From \eqref{e.Hpq}, $N^2 H_{p,q}(s)$ is a difference of the two terms.  The first is
\begin{equation} \label{e.cov->1} \begin{aligned} &\frac14\E[\Tr((A^{2N}_s)^p)\Tr((A^{2N}_s)^q)] \\
& \qquad\qquad = \frac14 \mathrm{Cov}[\Tr((A^{2N}_s)^p),\Tr((A^{2N}_s)^q)] + \frac14 \E[\Tr((A^{2N}_s)^p)]\E[\Tr((A^{2N}_s)^q)]. \end{aligned} \end{equation}
The second term is a sum
\[ -\frac12\sum_{\ell=1}^2 \E[\Tr((A^{2N}_s)^pP_\ell)\Tr((A^{2N}_s)^qP_\ell)]. \]
Let $R$ be the block rotation matrix of $\C^{2N}$ by $\frac{\pi}{2}$ in each factor of $\C^N$, so that $RP_1R^\ast = P_2$.  Since the distribution of $A^{2N}_s$ is invariant under rotations, it follows that
\[ \E[\Tr((A^{2N}_s)^pP_\ell)] \quad \text{and} \quad \E[\Tr((A^{2N}_s)^pP_\ell)\Tr((A^{2N}_s)^qP_\ell)] \]
do not depend on $\ell$ (as each is a conjugation-invariant polynomial function in $A_s^{2N}$).  In particular, the two terms in the $\ell$-sum are equal, and so the second term in $H_{p,q}(s)$ is
\begin{equation} \label{e.cov->2} \begin{aligned} & -\E[\Tr((A^{2N}_s)^pP_\ell)\Tr((A^{2N}_s)^qP_\ell)] \\
 & \qquad\qquad = -\mathrm{Cov}[\Tr((A^{2N}_s)^pP_1),\Tr((A^{2N}_s)^qP_1)] - \E[\Tr((A^{2N}_s)^pP_1)]\E[\Tr((A^{2N}_s)^qP_1)]. \end{aligned} \end{equation}
Moreover, since $\E[\Tr((A^{2N}_s)^pP_1)] = \E[\Tr((A^{2N}_s)^pP_2)]$ and $P_1+P_2 = I$, we have $\E[\Tr((A^{2N}_s)^pP_1)] = \frac12\E[\Tr((A^{2N}_s)^p)]$.  Thus, the last term in \eqref{e.cov->2} is $-\frac14\E[\Tr((A^{2N}_s)^p)]\E[\Tr((A^{2N}_s)^q)]$.  Combining this with \eqref{e.cov->1} then yields the result.  \end{proof}
Therefore, we are left to estimate these two covariance terms.  We do so by expanding them in terms of the Schur-Weyl representation.  Let $\gamma_n\in S_n$ be the full cycle and let $\gamma_{p,q}$ denote the double-cycle in $S_{p+q}$:
\[ \gamma_n = (1\,\cdots\,n), \qquad \gamma_{p,q} = (1\,\cdots\,p)(p+1\,\cdots\,p+q)\in S_{p+q}. \]
Then, for any matrix $A\in\M_{2N}$, \eqref{e.trace.cycles} gives
\[ \Tr(A^n) = \Tr[A^{\tensor n}\cdot \gamma_n], \qquad \Tr(A^p)\Tr(A^q) = \Tr[A^{\tensor p}\tensor A^{\tensor q}\cdot \gamma_{p,q}]. \]
It follows that
\begin{align}\nonumber &\mathrm{Cov}[\Tr((A^{2N}_s)^p),\Tr((A^{2N}_s)^q)] \\
\nonumber & \qquad\qquad = \E[\Tr((A^{2N}_s)^p)\Tr((A^{2N}_s)^q)]-\E[\Tr((A^{2N}_s)^p)]\E[\Tr((A^{2N}_s)^q)] \\
\nonumber & \qquad\qquad =\E\Tr[(A_s^{2N})^{\tensor p}\tensor (A_s^{2N})^{\tensor q}\cdot \gamma_{p,q}] - \E\Tr[(A_s^{2N})^{\tensor p}\cdot\gamma_p]\E\Tr[(A_s^{2N})^{\tensor q}\cdot\gamma_q] \\
\nonumber & \qquad\qquad = \Tr[\E((A_s^{2N})^{\tensor (p+q)})\cdot\gamma_{p,q}] - \Tr[\E((A_s^{2N})^{\tensor p})\tensor \E((A_s^{2N})^{\tensor p})\cdot \gamma_{p,q}] \\
\label{e.cov1.expansion}  & \qquad\qquad = \Tr\left(\left[\E((A_s^{2N})^{\tensor (p+q)})-\E((A_s^{2N})^{\tensor p})\tensor \E((A_s^{2N})^{\tensor p})\right]\cdot \gamma_{p,q}\right).
\end{align}
At the same time, using the fact that the projection $P_1$ are diagonal, we have for any matrix $A\in\M_{2N}$
\[ \Tr(A^pP_1)\Tr(A^qP_1) = \Tr[A^{\tensor p}\tensor A^{\tensor q}\cdot (P_1\tensor P_1)_{p,p+q}\cdot \gamma_{p,q}], \]
where we remind the reader that $(P_\ell\tensor P_\ell)_{i,j}$ references notation \eqref{e.notation.i,j}.  A similar calculation to the one above confirms that
\begin{align} \nonumber & \mathrm{Cov} [\Tr((A^{2N}_s)^pP_1),\Tr((A^{2N}_s)^qP_1)] \\
\label{e.cov2.expansion} = & \Tr\left(\left[\E((A_s^{2N})^{\tensor (p+q)})-\E((A_s^{2N})^{\tensor p})\tensor \E((A_s^{2N})^{\tensor p})\right]\cdot (P_1\tensor P_1)_{p,p+q}\cdot \gamma_{p,q}\right).
\end{align}
Thus, from \eqref{e.Hpq.cov}, \eqref{e.cov1.expansion}, and \eqref{e.cov2.expansion}, to estimate $H_{p,n-p}(s)$ we must understand the tensor
\begin{equation} \label{e.precov} \E((A_s^{2N})^{\tensor n})-\E((A_s^{2N})^{\tensor p})\tensor \E((A_s^{2N})^{\tensor (n-p)}). \end{equation}
To that end, we introduce some notation.  For $1\le i<j\le n$, define the linear operator $T_{i,j}$ on $(\C^{2N})^{\tensor n}$ by
\begin{equation} \label{e.Tij} T_{i,j} = 2(i\,j)\left[P_1\tensor P_1 + P_2\tensor P_2\right]_{i,j}. \end{equation}
Additionally, for $1\le p\le n$, we introduce the linear operators $\Phi_p$ and $\Psi_p$ as follows:
\begin{equation} \label{e.notation.PhiPsi} \Phi_p = -\frac{1}{2N}\sum_{\substack{ i\le i<j\le p,\text{ or}\\ p<i<j\le n}}(i\,j), \quad \Psi_p = -\frac{1}{2N}\sum_{\substack{ i\le i<j\le p,\text{ or}\\ p<i<j\le n}} T_{i,j}, \end{equation}
with the understanding that, when $p=n$, the sum is simply over $1\le i<j\le n$.

\begin{lemma} Let $p\in\{1,\ldots,n\}$, and let $0\le s\le t$.  Then
\begin{align} \label{e.Phi.meaning} \E[(U_s^{2N})^{\tensor p}]\tensor\E[(U_s^{2N})^{\tensor (n-p)}] &= e^{-\frac{ns}{2}}e^{s\Phi_p}, \\
\label{e.Psi.meaning} \E[(B_s^{2N})^{\tensor p}]\tensor\E[(B_s^{2N})^{\tensor (n-p)}] &= e^{-\frac{ns}{2}}e^{s\Psi_p}, \\
\label{e.spadediamond} \E[(A_s^{2N})^{\tensor p}]\tensor \E[(A_s^{2N})^{\tensor (n-p)}] &= e^{-\frac{nt}{2}} e^{(t-s)\Phi_p}e^{s\Psi_p}, \end{align}
where, in the case $p=n$, we interpret the $0$-fold tensor product as the identity as usual.
\end{lemma}

\begin{proof} Returning to the SDEs \eqref{e.SDE.tensor.n} and \eqref{e.SDE.tensor.n.2}, taking expectations we find that
\begin{align} \label{e.EODE2} \frac{d}{ds}\E[(U_s^{2N})^{\tensor n}] &= -\frac{n}{2}\E[(U_s^{2N})^{\tensor n}] - \frac{1}{2N}\E\left[(U_s^{2N})^{\tensor n}\cdot\sum_{1\le i<j\le n} (i\,j)\right], \\
\label{e.EODE1} \frac{d}{ds}\E[(B_s^{2N})^{\tensor n}] &= -\frac{n}{2}\E[(B_s^{2N})^{\tensor n}] - \frac{1}{2N}\E\left[(B_s^{2N})^{\tensor n}\cdot\sum_{1\le i<j\le n} T_{i,j}\right].
 \end{align}
Since the processes $B^{2N}$ and $U^{2N}$ start at the identity, it is then immediate to verify that the solutions to these ODEs are
\[ \E[(U_s^{2N})^{\tensor n}] = e^{-\frac{ns}{2}}\exp\left\{-\frac{s}{2N}\sum_{1\le i<j\le n}(i\,j)\right\}, \quad \text{and} \quad
\E[(B_s^{2N})^{\tensor n}] = e^{-\frac{ns}{2}}\exp\left\{-\frac{s}{2N}\sum_{1\le i<j\le n}T_{i,j}\right\}. \] 
Now, for the tensor product, we decompose
\[ \E[(U_s^{2N})^{\tensor p}]\tensor\E[(U_s^{2N})^{\tensor (n-p)}] = \left(\E[(U_s^{2N})^{\tensor p}]\tensor I^{\tensor (n-p)}\right)\cdot\left(I^{\tensor p}\tensor \E[(U_s^{2N})^{\tensor (n-p)}]\right). \]
We can express these expectations as in \eqref{e.EODE2}, provided we note that $(i\,j)$ now refers (alternatively) to action of $S_p$ and $S_{n-p}$ on $\M_{2N}^{\tensor n} = \M_{2N}^{\tensor p}\tensor \M_{2N}^{\tensor (n-p)}$, either trivially in the first factor or the second.  The result is that
\begin{align*} \E[(U_s^{2N})^{\tensor p}]\tensor I^{\tensor (n-p)} &= e^{-\frac{ps}{2}}\exp\left\{-\frac{s}{2N}\sum_{1\le i<j\le p} (i\,j)\right\} \\
I^{\tensor p}\tensor \E[(U_s^{2N})^{\tensor (n-p)}] &= e^{-\frac{(n-p)s}{2}}\exp\left\{-\frac{s}{2N}\sum_{p<i'<j'\le n} (i'\,j')\right\}.
\end{align*}
Note that all the $(i\,j)$ terms in the first sum commute with all the $(i'\,j')$ terms in the second sum (since $i<j<i'<j'$), and so, taking the product, we can combine to yield
\[ \E[(U_s^{2N})^{\tensor p}]\tensor\E[(U_s^{2N})^{\tensor (n-p)}] =e^{-\frac{ns}{2}}\exp\left\{-\frac{s}{2N}\sum_{\substack{1\le i<j\le p,\text{ or}\\ p<i<j\le n}} (i\,j)\right\} = e^{-\frac{ns}{2}}e^{s\Phi_p}, \]
verifying \eqref{e.Phi.meaning}.  An entirely analogous analysis proves \eqref{e.Psi.meaning}.  

Finally, using independence as in \eqref{e.G.product.indep} to factor

\begin{equation} \label{e.factor.4} \begin{aligned} &\E[(A_s^{2N})^{\tensor p}]\tensor \E[(A_s^{2N})^{\tensor (n-p)}] \\
& \hspace{1in}  = \left(\E[(U_{t-s}^{2N})^{\tensor p}]\tensor \E[(U_{t-s}^{2N})^{\tensor (n-p)}]\right)\cdot\left(\E[(B_s^{2N})^{\tensor p}]\tensor \E[(B_s^{2N})^{\tensor (n-p)}] \right) \end{aligned} \end{equation}
and substituting $s\mapsto t-s$ in \eqref{e.Phi.meaning}, \eqref{e.spadediamond} follows from \eqref{e.Psi.meaning} and \eqref{e.factor.4}.
\end{proof}

Hence, the desired quantity \eqref{e.precov} is computed, via \eqref{e.spadediamond}, as
\begin{equation} \label{e.precov.2} \E((A_s^{2N})^{\tensor n})-\E((A_s^{2N})^{\tensor p})\tensor \E((A_s^{2N})^{\tensor (n-p)}) = e^{-\frac{nt}{2}}\left[e^{(t-s)\Phi_n}e^{s\Psi_n}-e^{(t-s)\Phi_p}e^{s\Psi_p}\right]. \end{equation}
To compute this, we recall Duhamel's formula: for complex matrices $C,D$,
\[ e^C-e^D = \int_0^1 e^{(1-u)C}(C-D)e^{uD}\,du. \]
To apply this to \eqref{e.precov.2}, we add and subtract $e^{(t-s)\Phi_p}e^{s\Psi_n}$. Duhamel's formula then expresses this difference as follows:
\begin{align*} & e^{(t-s)\Phi_n}e^{s\Psi_n}-e^{(t-s)\Phi_p}e^{s\Psi_p} \\
=& \Big[e^{(t-s)\Phi_n}-e^{(t-s)\Phi_p}\Big]e^{s\Psi_n} + e^{(t-s)\Phi_p}\Big[e^{s\Psi_n}-e^{s\Psi_p}\Big] \\
=& \int_0^1 e^{(1-u)(t-s)\Phi_n}(t-s)(\Phi_n-\Phi_p)e^{u(t-s)\Phi_p}\,du\cdot e^{s\Psi_n} + e^{(t-s)\Phi_p}\int_0^1 e^{(1-u)s\Psi_n}s(\Psi_n-\Psi_p)e^{us\Psi_p}\,du \\
=& \int_0^{t-s} e^{(t-s-u)\Phi_n}(\Phi_n-\Phi_p) e^{u\Phi_p}e^{s\Psi_n}\,du + \int_0^s e^{(t-s)\Phi_p}e^{(s-u)\Psi_n}(\Psi_n-\Psi_p)e^{u\Psi_p}\,du
\end{align*}
where we have made the substitution $u\mapsto (t-s)u$ in the first integral and $u\mapsto su$ in the second.  Now, from \eqref{e.notation.PhiPsi},
\[ \Phi_p-\Phi_n = \frac{1}{2N}\sum_{1\le i\le p<j\le n} (i\,j), \quad \text{and} \quad \Psi_p-\Psi_n = \frac{1}{2N}\sum_{1\le i\le p<j\le n} T_{i,j}. \]
Hence, \eqref{e.precov.2} yields
\begin{equation} \begin{aligned} &\E((A_s^{2N})^{\tensor n})-\E((A_s^{2N})^{\tensor p})\tensor \E((A_s^{2N})^{\tensor (n-p)}) \\
= &\frac{e^{-\frac{nt}{2}}}{2N} \sum_{1\le i\le p<j\le n} \left( \int_0^{t-s} e^{(t-s-u)\Phi_n}(i\,j)e^{u\Phi_p}e^{s\Psi_n}\,du + \int_0^s e^{(t-s)\Phi_p}e^{(s-u)\Psi_n}T_{i,j}e^{u\Psi_p}\,du\right). \end{aligned} \end{equation}
We now reinterpret the exponentials in the integrals as expectations of the processes $U^{2N}$ and $B^{2N}$, using \eqref{e.Phi.meaning} and \eqref{e.Psi.meaning}.  The first integrand is
\begin{align*} e^{-\frac{nt}{2}} e^{(t-s-u)\Phi_n}(i\,j)e^{u\Phi_p}e^{s\Psi_n} 
= \E[(U^{2N}_{t-s-u})^{\tensor n}]\cdot (i\,j)\cdot \E[(U^{2N}_{u})^{\tensor p}]\tensor\E[(U^{2N}_{u})^{\tensor (n-p)}]\cdot \E[(B^{2N}_s)^{\tensor n}].
\end{align*}
By definition $U^{2N}$ and $B^{2N}$ are independent. Let us introduce two more copies $V^{2N},W^{2N}$ of $U^{2N}$ so the $\{U^{2N},V^{2N},W^{2N},B^{2N}\}$ are all independent.  Then this product may be expressed as the expectation of 
\begin{equation} \label{e.def.R} R^p_{i,j}(u;s,t)=(U^{2N}_{t-s-u})^{\tensor n}\cdot (i\,j)\cdot (V_u^{2N})^{\tensor p}\tensor (W_u^{2N})^{\tensor(n-p)}\cdot (B_s^{2N})^{\tensor n}. \end{equation}
Similarly, if we introduce independent copies $C^{2N}$ and $D^{2N}$ of $B^{2N}$ so that all the the processes $V^{2N}$, $W^{2N}$, $B^{2N}$, $C^{2N}$, $D^{2N}$ are independent, the second integrand can be expressed as the expectation of
\begin{equation} \label{e.def.Q} Q^p_{i,j}(u;s,t)=(V^{2N}_{t-s})^{\tensor p}\tensor (W^{2N}_{t-s})^{\tensor (n-p)}\cdot (B_{s-u}^{2N})^{\tensor n} \cdot T_{i,j}\cdot (C^{2N}_u)^{\tensor p}\tensor (D_u^{2N})^{\tensor (n-p)}. \end{equation}

To summarize, we have computed the following.

\begin{lemma} Let $0\le s\le t$ and $p\in\{1,\ldots,n\}$.  Then
\begin{equation} \label{e.precov.3}  \begin{aligned} &\E((A_s^{2N})^{\tensor n})-\E((A_s^{2N})^{\tensor p})\tensor \E((A_s^{2N})^{\tensor (n-p)}) \\
& \qquad\qquad = \frac{1}{2N} \sum_{1\le i\le p<j\le n} \left( \int_0^{t-s} \E[R^p_{i,j}(u;t,s)]\,du + \int_0^s \E[Q^p_{i,j}(u;t,s)]\,du\right). \end{aligned} \end{equation} \end{lemma}
We will now use this, together with \eqref{e.Hpq.cov}, \eqref{e.cov1.expansion}, and \eqref{e.cov2.expansion}, to estimate $H_{p,n-p}(s)$.  We estimate each of the two covariance terms separately, in the following two lemmas.

\begin{lemma} \label{l.final.est.1} For $p\in\{1,\ldots,n\}$ and $0\le s\le t$,
\begin{equation} \label{e.final.est.1} \left|\mathrm{Cov}[\Tr((A^{2N}_s)^p),\Tr((A^{2N}_s)^q)]\right| \le p(n-p)(t+3s). \end{equation}
\end{lemma}

\begin{proof} From \eqref{e.precov.3} together with \eqref{e.cov1.expansion}, the quantity whose modulus we wish to estimate is
\[ \sum_{1\le i\le p<j\le n}\left(\int_0^{t-s} \frac{1}{2N}\E\Tr\left[R^p_{i,j}(u;t,s)\cdot \gamma_{p,n-p}\right]\,du + \int_0^s \frac{1}{2N}\E\Tr\left[Q_{i,j}^p(u;t,s)\cdot\gamma_{p,n-p}\right]\,du\right). \]
For the first term, we note $(U_{t-s-u}^{2N})^{\tensor n}\gamma_{p,n-p} = \gamma_{p,n-p}(U_{t-s-u}^{2N})^{\tensor n}$ (the Schur-Weyl representation of any permutation in $S_n$ commutes with a matrix of the form $M^{\tensor n}$).  It follows from the trace property that
\begin{align*} \E\Tr[R^p_{i,j}(u;s,t)\cdot\gamma_{p,n-p}] &= \E\Tr[(V_u^{2N})^{\tensor p}\tensor (W_u^{2N})^{\tensor(n-p)}\cdot (B_s^{2N})^{\tensor n}\cdot \gamma_{p,n-p}\cdot (U^{2N}_{t-s-u})^{\tensor n}\cdot (i\,j)] \\
&=\E\Tr[ (V_u^{2N})^{\tensor p}\tensor (W_u^{2N})^{\tensor(n-p)}\cdot (B_s^{2N})^{\tensor n}\cdot (U^{2N}_{t-s-u})^{\tensor n}\cdot\gamma_{p,n-p}\cdot(i\,j)].
\end{align*}
Since $i\le p<j$, the permutation $\gamma_{p,n-p}(i\,j)$ is a single cycle.  Thus, by \eqref{e.trace.cycles}, the $\tensor n$-fold trace reduces to a trace of some $p,n,i,j$-dependent word in $V_u^{2N}$, $W_u^{2N}$, $B_s^{2N}$, and $U^{2N}_{t-s-u}$.  This word is a random element of $\U_{2N}$, and hence
\[ \frac{1}{2N}\E\Tr[R^p_{i,j}(u;s,t)\cdot\gamma_{p,n-p}]= \frac{1}{2N}\E\Tr[\text{a random matrix in }\U_{2N}] \text{ which $\therefore$ has modulus }\le 1. \]
Hence, the first integral is
\begin{equation} \label{e.final.est.1.1} \left|\int_0^{t-s}  \frac{1}{2N}\E\Tr\left[R^p_{i,j}(u;t,s)\cdot \gamma_{p,n-p}\right]\,du\right| \le (t-s). \end{equation}

For the second term, the fact that $T_{i,j} = 2(i\,j)[P_1\tensor P_1 + P_2\tensor P_2]_{i,j}$ only acts non-trivially in the $i,j$ factors, and $i\le p<j$, shows that (as above) $T_{i,j}$ commutes  with $(C^{2N}_u)^{\tensor p}\tensor (D_u^{2N})^{\tensor (n-p)}$.  Hence, we can express the second integrand as
\begin{equation} \label{e.final.est.1.long} \begin{aligned} &\E\Tr\left[Q_{i,j}^p(u;t,s)\cdot\gamma_{p,n-p}\right] \\
 = &\E\Tr\left[(V^{2N}_{t-s})^{\tensor p}\tensor (W^{2N}_{t-s})^{\tensor (n-p)}\cdot (B_{s-u}^{2N})^{\tensor n} \cdot T_{i,j}\cdot (C^{2N}_u)^{\tensor p}\tensor (D_u^{2N})^{\tensor (n-p)}\cdot \gamma_{p,n-p}\right] \\
 = &\E\Tr\left[(V^{2N}_{t-s})^{\tensor p}\tensor (W^{2N}_{t-s})^{\tensor (n-p)}\cdot (B_{s-u}^{2N})^{\tensor n} \cdot T_{i,j}\cdot \gamma_{p,n-p}\cdot (C^{2N}_u)^{\tensor p}\tensor (D_u^{2N})^{\tensor (n-p)}\right] \\
 =& 2\sum_{\ell=1}^2  \E\Tr\left[(C^{2N}_u)^{\tensor p}\tensor (D_u^{2N})^{\tensor (n-p)}\cdot(V^{2N}_{t-s})^{\tensor p}\tensor (W^{2N}_{t-s})^{\tensor (n-p)}\cdot (B_{s-u}^{2N})^{\tensor n} \cdot   (P_\ell\tensor P_\ell)_{i,j}(i\,j)\cdot \gamma_{p,n-p}\right]
\end{aligned} \end{equation}
where we have used the fact that $\gamma_{p,n-p}$ commutes with any matrix of the form $C^{\tensor p}\tensor D^{\tensor (n-p)}$ in the second equality, and then the trace property in the third equality.  As above, each of these terms reduces to a trace of a word, this time of the form
\[ 2\sum_{\ell=1}^2 \E\Tr(\mathcal{U} P_\ell \mathcal{V} P_\ell) \]
where $\mathcal{U}$ and $\mathcal{V}$ are random matrices in $\U_{2N}$ (depending on $p,n,i,j$).  Since $\|P_\ell\|\le 1$, the modulus of each term is $\le 2N$, giving an overall factor of $\le 8N$.  Combining with the $\frac{1}{2N}$ in the integral, this gives
\begin{equation} \label{e.final.est.1.2} \left|\int_0^s  \frac{1}{2N}\E\Tr\left[Q^p_{i,j}(u;t,s)\cdot \gamma_{p,n-p}\right]\,du\right| \le 4s. \end{equation}
Hence, from \eqref{e.final.est.1.1} and \eqref{e.final.est.1.2}, the modulus of the desired covariance is bounded by
\[ \sum_{1\le i\le p<j\le n} [(t-s)+2s] = p(n-p)(t+3s), \]
yielding \eqref{e.final.est.1}. \end{proof}

\begin{lemma} \label{l.final.est.2} For $p\in\{1,\ldots,n\}$ and $0\le s\le t$,
\begin{equation} \label{e.final.est.2} \left|\mathrm{Cov}[\Tr((A^{2N}_s)^pP_1),\Tr((A^{2N}_s)^qP_1)]\right| \le p(n-p)\left(t+3s\right). \end{equation}
\end{lemma}

\begin{proof} From \eqref{e.precov.3} together with \eqref{e.cov2.expansion}, the quantity whose modulus we wish to estimate is
\begin{equation} \label{e.final.est.2.0} \begin{aligned} \sum_{1\le i\le p<j\le n}\Big(\int_0^{t-s} \frac{1}{2N}\E\Tr&\left[R^p_{i,j}(u;t,s)\cdot(P_1\tensor P_1)_{p,n} \cdot \gamma_{p,n-p}\right]\,du \\
& + \int_0^s \frac{1}{2N}\E\Tr\left[Q_{i,j}^p(u;t,s)\cdot(P_1\tensor P_1)_{p,n}\cdot\gamma_{p,n-p}\right]\,du\Big). \end{aligned} \end{equation}
For the first term, we expand the integrand, commuting $(i\,j)$ past $(U^{2N}_{t-s-u})^{\tensor n}$ as in the proof of Lemma \ref{l.final.est.1}, to give
\[ \frac{1}{2N}\E\Tr\left[(U^{2N}_{t-s-u})^{\tensor n}\cdot (V_u^{2N})^{\tensor p}\tensor (W_u^{2N})^{\tensor(n-p)}\cdot (B_s^{2N})^{\tensor n}\cdot (P_1\tensor P_1)_{p,p+q}\cdot\gamma_{p,n-p}\cdot (i\,j)\right]. \]
As above, since $\gamma_{p,n-p}\cdot(i,j)$ is a single cycle, this trace reduces to a trace over $\C^{2N}$, of a the form
\[ \frac{1}{2N}\E\Tr[\mathcal{U}'P_1\mathcal{V}'P_1] \]
where $\mathcal{U}'$ and $\mathcal{V}'$ are random unitary matrices in $\U_{2N}$ composed of certain $i,j,p,n$-dependent words in $U^{2N}_{t-s-u}$, $V^{2N}_u$, $W^{2N}_u$, and $B^{2N}_s$.  As $\|P_1\|\le 1$, it follows that this normalized trace is $\le 1$, and so the first integral in \eqref{e.final.est.2.0} is $\le (t-s)$ in modulus, as in \eqref{e.final.est.1.1}.

Similarly, we expand the second term as in \eqref{e.final.est.1.long}, which gives the sum over $\ell\in\{1,2\}$ of the expected normalized trace of 
\[ (V^{2N}_{t-s})^{\tensor p}\tensor (W^{2N}_{t-s})^{\tensor (n-p)}\cdot (B_{s-u}^{2N})^{\tensor n} \cdot \gamma_{p,n-p}(i,j) (P_\ell\tensor P_\ell)_{i,j}\cdot (C^{2N}_u)^{\tensor p}\tensor (D_u^{2N})^{\tensor (n-p)}\cdot(P_1\tensor P_1)_{p,n}. \]
As in the above cases, since $\gamma_{p,n-p}(i,j)$ is a single cycle, this is equal to a single trace $\Tr(A_1\cdots A_n)$ where $A_1,\ldots,A_n$ belong to the set $\{V_{t-s}^{2N},W_{t-s}^{2N},B^{2N}_{s-u},C^{2N}_u,D^{2N}_u,P_\ell,P_1\}$.  As each of these is either a random unitary matrix or a projection, it follows that the expected normalized trace has modulus $\le 1$, and so the $\frac1N$-weighted sum of $2$ terms, each of modulus $\le 2N$, gives a contribution no bigger than $4$.  The remainder of the proof is exactly as the end of the proof of Lemma \ref{l.final.est.1}. \end{proof}

Finally, we are ready to conclude this section.

\begin{proof}[Proof of Proposition \ref{p.moment.rate.Cauchy}] From \eqref{e.nu.intermsof.Hpq}, we have
\[ |\nu_t^N(n)-\nu_t^{2N}(n)| \le \frac{n}{2}\sum_{p=1}^{n-1} \int_0^t |H_{p,n-p}(s)|\,ds. \]
Lemma \ref{l.Hpq.cov} then gives
\[ |H_{p,q}(s)| \le \frac{1}{4N^2}\left|\mathrm{Cov}[\Tr((A^{2N}_s)^p),\Tr((A^{2N}_s)^q)]\right| + \frac{1}{N^2}\left|\mathrm{Cov}[\Tr((A^{2N}_s)^pP_1),\Tr((A^{2N}_s)^qP_1)]\right|. \]
Combining this with \eqref{e.final.est.1} and \eqref{e.final.est.2} therefore yields
\[ |H_{p,q}(s)| \le \frac{p(n-p)}{N^2}\cdot \frac{5}{4}(t+3s). \]
Integration then gives
\[ |\nu_t^N(n)-\nu_t^{2N}(n)| \le \frac{n}{2}\sum_{p=1}^{n-1} \frac{p(n-p)}{N^2}\cdot \frac{25}{8}t^2 = \frac{25t^2n}{16N^2}\sum_{p=1}^{n-1} p(n-p). \]
The sum over $p$ has the exact value $\frac16(n^3-n) \le \frac{n^3}{6}$.  The blunt estimate $\frac{25}{96}<\frac34$ yields the result.
\end{proof}


\newpage

\section{Strong Convergence \label{section strong convergence}}

In this section, we prove Theorem \ref{t.main.2}.  We begin by showing (in Section \ref{section Marginals UBM GUE}) that the eigenvalues of the unitary Brownian motion at a fixed time converge to their ``classical locations'', and we use this to prove that the unitary Brownian motion can be uniformly approximated by a function of a Gaussian unitary ensemble (for time $t<4$).  We then use this, together with Male's Theorem \ref{t.Male}, to prove Theorem \ref{t.main.2}.

\subsection{Marginals of Unitary Brownian Motion and Approximation by $\mathrm{GUE}^N$ \label{section Marginals UBM GUE}}

\begin{remark} \label{r.subseq} Throughout this section, we will regularly use the following elementary fact: if $(a^N)_{N=1}^\infty$ is a sequence that possesses a convergent subsequence, and if {\em all} convergent subsequences have the same limit $r$, then $\lim_{N\to\infty} a^N = r$. \end{remark}

We begin with the following general result on convergence of empirical measures.  As usual, for a probability measure $\mu$ on $\R$, the cumulative distribution function $F_\mu$ is the nondecreasing function $F_\mu(x) = \mu((-\infty,x])$.  If $\mu$ has a density $\rho$, we may abuse notation and write $F_\mu = F_\rho$.

\begin{proposition} \label{p.conv.point.process} For each $N\in\N$, let $(x^N_k)_{k=1}^N$ be points in $\R$ with $x^N_1\le \cdots \le x^N_N$.  Let $\mu^N = \frac1N\sum_{k=1}^N \delta_{x^N_k}$ be the associated empirical measures.  Suppose the following hold true.
\begin{itemize}
\item[(1)] There is a compact interval $[a_-,a_+]$ and a continuous probability density $\rho$ with $\supp\rho = [a_-,a_+]$, so that, with $\mu(dx) = \rho(x)\,dx$, we have $\mu^N\rightharpoonup \mu$ weakly as $N\to\infty$.
\item[(2)] $x^N_1 \to a_-$ and $x^N_N\to a_+$ as $N\to\infty$.
\end{itemize}
For $r\in[0,1]$, define $x^\ast(r) =F_\mu^{-1}(r)$ if $r\in(0,1)$, and $x^\ast(0)=a_-$, $x^\ast(1) = a_+$.  Then, for any sequence $(k(N))_{N=1}^\infty$ with $k(N)\in\{1,\ldots,N\}$ and $\displaystyle{\lim_{N\to\infty}}k(N)/N=r$, we have 
\[ \lim_{N\to\infty} x^N_{k(N)} = x^\ast(r). \]
\end{proposition}

\begin{proof} Since $\nu^N\rightharpoonup\nu$, it follows that $F_{\mu^N}(x)\to F_\mu(x)$ for any $x$ that is a continuity point of $F_\mu$; since $F_\mu$ is continuous everywhere, we therefore have $F_{\mu^N}\to F$ pointwise on $\R$.  As $\mu$ is compactly supported, it now follows by a variant of Dini's theorem (cf.\ \cite[Problem 127 Chapter II]{PolyaSzego1998}) that $F_{\mu^n}\to F$ uniformly.  Now, let $k(N)$ be a sequence as stated, and set
\[ y^N = x^N_{k(N)}, \quad \text{and} \quad a^N = F_{\mu}(y^N).  \]
Since $1\le k(N)\le N$, and since the points $x^N_k$ are ordered, we have $x^N_1\le y^N \le x^N_N$.  Therefore item (2) shows that $(y^N)$ is a bounded sequence, and hence possesses convergent subsequences.  Fix any convergent subsequence $(y^{N_m})_{m=1}^\infty$; then, since $F_\mu$ is continuous, $(a^{N_m})$ is also convergent.  Since $F_{\mu^{N_m}}\to F_\mu$ uniformly, it follows that
\[ |a^{N_m} - F_{\mu^{N_m}}(y^{N_m})| = |F_\mu(y^{N_m}) - F_{\mu^{N_m}}(y^{N_m})| \to 0 \quad\text{ as }\; m\to\infty. \]
Note that, for any $n\in\N$, since $F_{\mu^n}$ is the normalized counting measure of the points $(x^n_k)_{k=1}^N$, we have
\[ F_{\mu^n}(y^n) = F_{\mu^n}(x^n_{k(n)}) = \frac{1}{n}\sum_{k=1}^n \1\{x^n_k \le x^n_{k(n)}\} = \frac{k(n)}{n}\to r \quad \text{as} \; n\to\infty. \]
Hence, we have shown that $a^{N_m}\to r$.  The limit $r$ does not depend on the subsequence, and so we conclude (cf.\ Remark \ref{r.subseq}) that
\[ r = \lim_{N\to\infty} a^N = \lim_{N\to\infty} F_\mu(x^N_{k(N)}). \]

Now, let $0<r<1$.  Since $\mu$ has a positive density on $[0,1]$, the function $F_\mu$ is continuous and strictly increasing on $[a_-,a_+]$, and so has a strictly increasing continuous inverse function $F_{\mu}^{-1}$ on $[0,1]$.  Thus, it follows that $x^N_{k(N)} = F_\mu^{-1}(a^N) \to F_\mu^{-1}(r) = x^\ast(r)$ as claimed.

On the other hand, suppose $r=0$.  Then for $0<\e<1$, for all large $N$, $k(N)<\e N$.  Thus $x^N_1 \le y^N \le x^N_{\e N}$.  By the preceding paragraph, we know $x^N_{\e N}\to x^\ast(\e)$, and by (2) $x^N_1\to a_- = x^\ast(0)$; thus $x^\ast(0)\le y^N\le x^\ast(\e)$ for all large $N$.  Since $x^\ast(\e) = F_\mu^{-1}(\e)$ and $F_\mu^{-1}$ is continuous on $[0,1]$, it follows that $x^\ast(\e)\to x^\ast(0)$ as $\e\downarrow 0$.  So $x^N_{k(N)} = y^N\to x^\ast(0)$ by the Squeeze Theorem.  An analogous argument completes the proof in the case $r=1$.  \end{proof}

For example, if $(x^N_k)_{k=1}^N$ are the ordered eigenvalues of a $\mathrm{GUE}^N$, then Wigner's law (and the corresponding hard edge theorem) show that the empirical spectral distribution satisfies the conditions of Proposition \ref{p.conv.point.process} almost surely, where the limit measure is the semicircle law $\mu=\sigma_1\equiv \frac{1}{2\pi}\sqrt{(4-x^2)_+}\,dx$.  In particular, when $\frac{k(N)}{N}\to r$, we have $x^N_{k(N)} \to F_{\sigma_1}^{-1}(r)$.  These values are sometimes called the {\em classical locations} of the eigenvalues.  In the case of a $\mathrm{GUE}^N$, much more is known; for example, \cite{Gustavsson2005} showed that the eigenvalues have variance of $O(\frac{\log N}{N^2})$ in the bulk and $O(N^{-4/3})$ at the edge, and further standardizing them, their limit distribution is Gaussian.  For our purposes, the macroscopic statement of Proposition \ref{p.conv.point.process} will suffice.

Now, fix $t\in [0,4)$.  From Theorem \ref{t.nu.t}, the law $\nu_t$ of the free unitary Brownian motion $u_t$ has an analytic density $\varrho_t$ supported on a closed arc strictly contained in $\U_1$, and has the form $\varrho_t(e^{ix}) = \rho_t(x)$ for some strictly positive continuous probability density function $\rho_t\colon(-\pi,\pi)\to\R$ which is symmetric about $0$ and supported in a symmetric interval $[-a(t),a(t)]$ where $a(t) = \frac12\sqrt{t(4-t)}+\arccos(1-t/2)$, cf.\ \eqref{e.supp.nu_t}.  For $0< r< 1$, define that {\em classical locations} $\upsilon^\ast(t,r)$ of the eigenvalues of unitary Brownian motion as follows:
\[ \upsilon^\ast(t,r) = \exp\left(iF_{\rho_t}^{-1}(r)\right), \]
and also set $\upsilon^\ast(t,0) = e^{-ia(t)}$ and $\upsilon^\ast(t,1) = e^{ia(t)}$.

\begin{corollary} \label{c.conv.point.process} Let $0\le t<4$, and let $V_t^N$ be a random unitary matrix distributed according to the heat kernel on $\U_N$ at time $t$ (i.e.\ equal in distribution to the $t$-marginal of the unitary Brownian motion $U_t^N$).  Enumerate the eigenvalues of $V_t^N$ as $\upsilon^N_1(t),\ldots,\upsilon^N_N(t)$, in increasing order of complex argument in $(-\pi,\pi)$.  Then for any sequence $(k(N))_{N=1}^\infty$ with $k(N)\in\{1,\ldots,N\}$ and $\displaystyle{\lim_{N\to\infty}}\textstyle{k(N)/N=r}$, we have
\[ \lim_{N\to\infty} \upsilon^N_{k(N)}(t) = \upsilon^\ast(t,r) \;\; a.s. \]
\end{corollary}

\begin{proof} Let $x^N_k(t) = -i\log \upsilon^N_k(t)$, where we use the branch of the logarithm cut along the negative real axis.  Note: by Theorem \ref{t.main.1}, for sufficiently large $N$, $\upsilon^N_k(t)$ are outside a $t$-dependent neighborhood of $-1$, and so the log function is continuous.  The empirical law of $\{\upsilon^N_k(t)\colon 1\le k\le N<\infty\}$ converges weakly a.s.\ to $\nu_t$ (cf.\ \eqref{e.nu_t=limit}), and so by continuity, the empirical measure of $\{x^N_k(t)\colon 1\le k\le N<\infty\}$ converges a.s.\ to the density $\rho_t$.  Moreover, Theorem \ref{t.main.1} shows that $\upsilon^N_1(t)\to e^{-ia(t)}$ and $\upsilon^N_N(t)\to e^{ia(t)}$ a.s., and so $x^N_1(t)\to -a(t)$ while $x^N_N(t)\to a(t)$ a.s.  Hence, by Proposition \ref{p.conv.point.process}, $x_{k(N)}^N(t)\to F_{\rho_t}^{-1}(r)$ whenever $k(N)/N\to r\in (0,1)$, while the sequence converges to $\pm a(t)$ when $r=0,1$.  Taking $\exp(i\cdot)$ of these statements yields the corollary.  \end{proof}

Now, let us combine this result with the comparable one for the $\mathrm{GUE}^N$.  Let $g_t\colon\R\to\R$ be given by $g_t = F_{\rho_t}^{-1}\circ F_{\sigma_1}$; this is an increasing, continuous map that pushes $\sigma_1$ forward to $\rho_t$.  Define $f_t\colon\R\to\U_1$ by
\begin{equation} \label{e.def.ft} f_t = \exp(ig_t), \qquad \therefore \; \nu_t = (f_t)_\ast(\sigma_1). \end{equation}
The main result of this section is that, rather than just pushing the semicircle law forward to the law of free unitary Brownian motion, $g_t$ in fact pushes a $\mathrm{GUE}^N$ forward, asymptotically, to $V_t^N$ (for fixed $t\in[0,4)$).

\begin{proposition} \label{p.GUE.approx.Ut} Let $0\le t<4$, and let $V_t^N$ be a random unitary matrix distributed according to the heat kernel on $\U_N$ at time $t$ (i.e.\ equal in distribution to the $t$-marginal of the unitary Brownian motion $U_t^N$).  There exists a self-adjoint random matrix $X^N$ with the following properties:
\begin{itemize}
\item[(1)] $X^N$ is a $\mathrm{GUE}^N$.
\item[(2)] The eigenvalues of $X^N$ are independent from $V_t^N$, and $\{V^N_t, X^N\}$ have the same eigenvectors.
\item[(3)] $\|f_t(X^N)-V^N_t\|_{\M_N}\to 0 \;\; a.s.$ as $N\to\infty$.
\end{itemize}
\end{proposition}

\begin{proof} Let $(\upsilon_k^N(t))_{k=1}^N$ denote the eigenvalues of $V^N_t$ in order of increasing argument in $(-\pi,\pi)$, as in Corollary \ref{c.conv.point.process}.  It is almost surely true that $\upsilon_1^N(t)<\cdots<\upsilon_N^N(t)$, and so we work in this event only.  Let $\mathscr{E}_k^N$ denote the eigenspace of the eigenvalue $\upsilon_k^N(t)$.  This space has complex dimension $1$ a.s., and so we may select a unit length vector $E_k^N$ from this space, with phase chosen uniformly at random in $\U_1$, independently for each of $E_1^N,\ldots,E_N^N$.  Then, by orthogonality of distinct eigenspaces, the random matrix $E^N = [\begin{smallmatrix} E^N_1 & \cdots & E^N_N \end{smallmatrix}]$ is in $\U_N$; what's more, since the distribution of $V_t^N$ is invariant under conjugation by unitaries, it follows that $E^N$ is Haar distributed on $\U_N$.

Now, for each $N$, fix a random vector $\boldsymbol\mu^N = (\mu_1^N,\ldots,\mu_N^N)$ independent from $V_t^N$ with joint density $f_{\boldsymbol\mu^N}(x_1,\ldots,x_N)$ equal to the known joint density of eigenvalues of a $\mathrm{GUE}^N$, i.e.\ proportional to
\[ f_{\boldsymbol\mu^N}(x_1,\ldots,x_n) \sim \prod_{j=1}^N e^{-\frac{N}{2} x_j^2}\prod_{1\le j<k\le N} |x_j-x_k|^2. \]
Then we define
\[ X^N \equiv E^N \left[\begin{array}{cccc} \mu_1^N & 0 & \cdots & 0 \\ 0 & \mu_2^N & \cdots & 0 \\ \vdots & \vdots & \ddots & \vdots \\ 0 & 0 & \cdots & \mu^N_N \end{array}\right] (E^N)^\ast. \]
It is well known (cf. \cite{AGZ,Mehta}) that the distribution of $X^N$ is the $\mathrm{GUE}^N$, verifying item (1).  Item (2) holds by construction of $X^N$.  It remains to see that (3) holds true.  Note, since operator norm is invariant under unitary conjugation, we simply have
\begin{equation} \label{e.max.eig.sep} \|f_t(X^N)-V_t^N\|_{\M_N} = \max_{1\le k\le N} |f_t(\mu_k^N) - \upsilon_j^N| \equiv |f_t(\mu_{k(N)}^N) - \upsilon^N_{k(N)}| \end{equation}
where $k(N)$ is an index in $\{1,\ldots,N\}$ that achieves the maximum.  The quantity in \eqref{e.max.eig.sep} has modulus $\le 2$ for each $N$, and so there are convergent subsequences.  We will show that all convergent subsequences converge to $0$, which proves (3) (cf.\ Remark \ref{r.subseq}).

Hence, reindexing as necessary, we assume that $|f_t(\mu_{k(N)}^N)-\upsilon_{k(N)}^N|$ converges.  Now, the sequence $k(N)/N$ is contained in $(0,1]$, and hence it has a convergent subsequence with some limit $r\in[0,1]$.  Again reindexing, we still denote this subsequence as $k(N)$.  By Proposition \ref{p.conv.point.process} in the case of the $\mathrm{GUE}^N$-eigenvalues, we know that $\mu_{k(N)}^N\to F_{\sigma_1}^{-1}(r)$ (where this is interpreted to equal $\pm 2$ when $r=0,1$).  Then by definition (and continuity) of $f_t$, $f_t(\mu_{k(N)}^n) \to \exp(iF_{\rho_t}^{-1}(r)) = \upsilon^\ast(t,r)$.  By Corollary \ref{c.conv.point.process}, we have $\upsilon_{k(N)}^N(t)\to \upsilon^\ast(t,r)$ as well.  This shows the limit of the difference is $0$.

Thus, given any convergent subsequence of $\|f_t(X^N)-V_t^N\|_{\M_N}$, there is a further subsequence that converges to $0$.  It follows that every convergent subsequence of $\|f_t(X^N)-V_t^N\|_{\M_N}$ has limit $0$, and so we conclude that it converges to $0$, as claimed.  \end{proof}

\subsection{Strong Convergence of the Process $(U^N_t)_{t\ge 0}$}

Since the Gaussian unitary ensemble is selfadjoint, we may extend Male's Theorem \ref{t.Male} to continuous functions in independent $\mathrm{GUE}^N$s.

\begin{lemma} \label{l.f(X)} Let $\mx{A}^N = (A^N_1,\ldots,A^N_n)$ be a collection of random matrix ensembles that converges strongly to some $\mx{a}=(a_1,\ldots,a_n)$ in a $W^\ast$-probability space $(\A,\t)$.  Let $\mx{X}^N = (X^N_1,\ldots,X^N_k)$ be independent Gaussian unitary ensembles independent from $\mx{A}^N$, and let $\mx{x} = (x_1,\ldots,x_k)$ be freely independent semicircular random variables in $\A$ all free from $\mx{a}$.  Let $\mx{f}=(f_1,\ldots,f_k)\colon\R\to\C^k$ be continuous functions, and let $\mx{f}(\mx{X}^N) = (f_1(X^N_1),\ldots,f_k(X^N_k))$ and $\mx{f}(\mx{x}) = (f_1(x_1),\ldots,f_k(x_k))$.  Then $(\mx{A}^N,\mx{f}(\mx{X}^N))$ converges strongly to $(\mx{a},\mx{f}(\mx{x}))$.
\end{lemma}

\begin{proof} We begin with the case $k=1$.  If $p$ is any polynomial, by Theorem \ref{t.Male}, $(\mx{A}^N,p(X^N_1))$ converges strongly to $(\mx{a},p(x_1))$ by definition.  Now, let $\e>0$, and fix a noncommutative polynomial $P$ in $n+1$ variables.  Then $P(\mx{a},y)$ is a finite sum of monomials, each of the form
\[ Q_0(\mx{a})yQ_1(\mx{a})y\cdots Q_{d-1}(\mx{a})yQ_d(\mx{a}) \]
for some noncommutative polynomials $Q_0,\ldots,Q_d$ and nonnegative integers $d$.  Let $d_P$ be the ``degree'' of $P$: the maximum number of $Q_k(\mx{a})$ terms that appears in any monomial in the above expansion of $P(\mx{a},y)$.  Let $M = 1+ $ the sum of all the products $\|Q_0(\mx{a})\|\cdots \|Q_d(\mx{a})\|$ over all monomial terms appearing in $P$.  By the Weierstrass approximation theorem, there is a polynomial $p$ in one variable so that
\begin{equation} \label{e.p-f1} \|p-f_1\|_{L^\infty[-2,2]} < \frac{\e}{8d_PM(1+\|f_1\|_{L^\infty[-2,2]})^{d_P}}. \end{equation}
It follows that, for small enough $\e>0$, we also have $\|p\|_{L^\infty[-2,2]} \le 1+\|f_1\|_{L^\infty[-2,2]}$.
Now we break up the difference in the usual manner,
\begin{equation} \label{e.three-terms}\begin{aligned}  \left|\|P(\mx{A}^N,f_1(X^N_1))\| - \|P(\mx{a},f_1(x_1))\|\right|
 &\le   \left|\|P(\mx{A}^N,f_1(X^N_1))\| - \|P(\mx{A}^N,p(X^N_1))\|\right| \\
& \qquad  + \left|\|P(\mx{A}^N,p(X^N_1))\| - \|P(\mx{a},p(x_1))\|\right| \\
& \qquad  + \left|\|P(\mx{a},p(x_1))\| - \|P(\mx{a},f(x_1))\|\right|.
\end{aligned} \end{equation}
By the known strong convergence of $(\mx{A}^N,p(X_1^N))$ to $(\mx{a},p(x_1))$, the middle term is $<\frac{\e}{4}$ for all sufficiently large $N$.  For the first and third terms, we use the reverse triangle inequality; in the third term this gives
\[  \left|\|P(\mx{a},p(x_1))\| - \|P(\mx{a},f(x_1))\|\right| \le \|P(\mx{a},p(x_1))-P(\mx{a},f(x_1))\|. \]
Let $y=p(x_1)$ and $z=f_1(x_1)$.  We may estimate the norm of the difference using the triangle inequality summing over all monomial terms; then we have a sum of terms of the form
\begin{equation} \label{e.int.term.1} \|Q_0(\mx{a})yQ_1(\mx{a})y\cdots Q_{d-1}(\mx{a})yQ_d(\mx{a}) - Q_0(\mx{a})zQ_1(\mx{a})z\cdots Q_{d-1}(\mx{a})zQ_d(\mx{a})\|. \end{equation}
By introducing intermediate mixed terms of the form $Q_0(\mx{a})y\cdots Q_{k-1}(\mx{a})yQ_k(\mx{a})z\cdots Q_{d-1}(\mx{a})zQ_d(\mx{a})$ to give a telescoping sum, we can estimate the term in \eqref{e.int.term.1} by
\begin{equation} \label{e.est.terms.2} \|Q_0(\mx{a})\|\cdots\|Q_d(\mx{a})\|\sum_{k=1}^d \|y\|^{k-1}\|z\|^{d-k}\|y-z\|. \end{equation}
Since $\|y\|=\|p(x_1)\| = \|p\|_{L^\infty[-2,2]} \le 1+\|f_1\|_{L^\infty[-2,2]}$ and $\|z\| = \|f_1(x_1)\| \le 1+\|f_1\|_{L^\infty[-2,2]}$, each term in the previous sum is bounded by
\[ \|y\|^{k-1}\|z\|^{d-k} \le (1+\|f_1\|_{L^\infty[-2,2]})^{d-1} \le  (1+\|f_1\|_{L^\infty[-2,2]})^{d_P}. \]
Since $\|y-z\| = \|p-f_1\|_{L^\infty[-2,2]}$, combining this with \eqref{e.p-f1} shows that the third term in \eqref{e.three-terms} is $<\frac{\e}{4}$ for all large $N$.

The first term in \eqref{e.three-terms} is handled in an analogous fashion, with the caveat that the prefactor in \eqref{e.est.terms.2} is replaced by $\|Q_0(\mx{A}^N)\|\cdots\|Q_d(\mx{A}^N)\|$.  Here we use the assumption of strong convergence of $\mx{A}^N\to \mx{a}$ to show that, for all sufficiently large $N$,
\[\|Q_0(\mx{A}^N)\|\cdots\|Q_d(\mx{A}^N)\| \le \max\{1,2\cdot \|Q_0(\mx{a})\|\cdots\|Q_d(\mx{a})\|\}. \]
Then we see that the first term in \eqref{e.three-terms} is $<\frac{\e}{2}$ for all large $N$, and so we have bounded the sum $<\e$ for all large $N$, concluding the proof of the lemma in the case $k=1$.

Now suppose we have verified the conclusion of the lemma for a given $k$.  We proceed by induction.  Taking $(\mx{A}^N,f_1(X^N_1),\ldots,f_k(X^N_k))$ as our new input vector, since $f_{k+1}(X^N_{k+1})$ is independent from all previous terms, the induction hypothesis and the preceding argument in the case $k=1$ give strong convergence of the augmented vector $(\mx{A}^N,f_1(X^N_1),\ldots,f_k(X^N_k),f_{k+1}(X^N_{k+1}))$ as well.  Hence, the proof is complete by induction. \end{proof}

This finally brings us to the proof of Theorem \ref{t.main.2}.

\begin{proof}[Proof of Theorem \ref{t.main.2}] As above, let $\mx{A}^N = (A^N_1,\ldots,A^N_n)$ and let $\mx{a} = (a_1,\ldots,a_n)$ be the strong limit.  By reindexing the order of the variables in the noncommutative polynomial $P$ appearing in the definition of strong convergence, it suffices to prove the theorem in the case of time-ordered entries: $U^N_{t_1},\ldots,U^N_{t_k}$ with $t_1\le t_2\le\cdots\le t_k$.  What's more, we may assume without loss of generality that the time increments $s_1=t_1,s_2=t_2-t_1,\ldots,s_k=t_k-t_{k-1}$ are all in $[0,4)$.  Indeed, if we know the theorem holds in this case, then for a list of ordered times with some gaps $4$ or larger, we may introduce intermediate times until all gaps are $<4$; then the restricted theorem implies strong convergence for this longer list of marginals, which trivially implies strong convergence for the original list.

Now, set $V_{s_1}^N = U^N_{t_1}$, and $V_{s_j}^N = (U^N_{t_{j-1}})^\ast U^N_{t_j}$ for $2\le j\le k$.  As discussed in Section \ref{section UBM def}, these increments of the process are independent, and $V_{s_j}^N$ has the same distribution as $U_{s_j}^N$.  Hence, by Proposition \ref{p.GUE.approx.Ut}, there are $k$ independent $\mathrm{GUE}^N$s $X^N_1,\ldots,X^N_k$, and continuous functions $f_{s_j}\colon\R\to\C$, so that $\|f_{s_j}(X^N_j) - V^N_{s_j}\|_{\M_N}\to 0$ as $N\to\infty$.  Since the $V_{s_j}^N$ are all independent from $\mx{A}^N$, so are the $X^N_j$.  Hence, by Lemma \ref{l.f(X)}, taking $x_1,\ldots,x_k$ freely independent semicircular random variables all free from $\mx{a}$, it follows that
\[ (\mx{A}^N,f_{s_1}(X^N_1),\ldots,f_{s_k}(X^N_k))\text{ converges strongly to }(\mx{a},f_{s_1}(x_1),\ldots,f_{s_k}(x_k)). \]
By the definition of the mapping $f_s$ (cf.\ \eqref{e.def.ft}), $f_{s_j}(x_j)$ has distribution $\nu_{s_j}$, and as all variables in sight are free, $(\mx{a},f_{s_1}(x_1),\ldots,f_{s_k}(x_k))$ has the same distribution as $(\mx{a},v_{s_1},\ldots,v_{s_k})$ where $(v_s)_{s\ge 0}$ is a free unitary Brownian motion, freely independent from $\mx{a}$.

It now follows, since $\|f_{s_j}(X_j^N)-V_{s_j}^N\|_{\M_N}\to 0$, that
\[ (\mx{A}^N,V_{s_1}^N,\ldots,V_{s_k}^N)\text{ converges strongly to }(\mx{a},v_{s_1},\ldots,v_{s_k}). \]
(The proof is very similar to the proof of Lemma \ref{l.f(X)}.)  Finally, we can recover the original variables $U^N_{t_j} = V^N_{s_1}V^N_{s_2}\cdots V^N_{s_j}$. Therefore
\[ (\mx{A}^N,U_{t_1}^N,\ldots,U_{t_k}^N)\text{ converges strongly to }(\mx{a},v_{s_1},v_{s_1}v_{s_2}\ldots,v_{s_1}v_{s_2}\cdots v_{s_k}). \]
The discussion at the end of Section \ref{section free prob} shows that $(v_{s_1},v_{s_1}v_{s_2},\ldots,v_{s_1}v_{s_2}\cdots v_{s_k})$ has the same distribution as $(u_{t_1},u_{t_2},\ldots,u_{t_k})$ where $(u_t)_{t\ge 0}$ is a free unitary Brownian motion in the $W^\ast$-algebra generated by $(v_s)_{s\ge 0}$, and is therefore freely independent from $\mx{a}$.  This concludes the proof.  \end{proof}

\newpage

\section{Application to the Jacobi Process \label{section application}}

In this final section we combine our main Theorem \ref{t.main.2} with some of the results of our earlier paper \cite{CK2014}, to show that the Jacobi process (cf.\ \eqref{e.Jacobi.ens.def} and \eqref{e.Jacobi.process}) has hard edges that evolve with finite propagation speed.

There are three classical Hermitian Gaussian ensembles that have been well studied.  The first is the Gaussian Unitary Ensemble described in detail above, whose analysis was initiated by Wigner \cite{Wigner} and began random matrix theory.  The second is the {\em Wishart Ensemble}, also known (through its applications in statistics) as a {\em sample covariance matrix}.  Let $a\ge1$, and let $X=X^N$ be an $N\times \lfloor aN \rfloor$ matrix all of whose entries are independent normal random variables of variance $\frac1N$; then $W = XX^\ast$ is a Wishart ensemble with parameter $a$.  As $N\to\infty$, its empirical spectral distribution converges almost surely to a law known as the {\em Marchenko-Pastur distribution}; this was proved in \cite{MarchenkoPastur1967}.  As with the Gaussian Unitary Ensemble, it also has a hard edge, and the largest eigenvalue when properly renormalized has the Tracy-Widom law.

The third Hermitian Gaussian ensemble is the {\em Jacobi Ensemble}.  Let $W_a$ and $W_b'$ be independent Wishart ensembles of parameters $a,b\ge1$.  Then it is known that $W_a+W_b'$ is a Wishart ensemble of parameter $a+b$, and is a.s.\ invertible (cf.\ \cite[Lemma 2.1]{Collins2005}).  The associated Jacobi Ensemble is
\begin{equation} \label{e.Jacobi.ens.def} J=J_{a,b} = (W_a+W_b')^{-\frac12}W_a(W_a+W_b')^{-\frac12} . \end{equation}
Such matrices have been studied in the statistics literature for over thirty years; they play a key role in MANOVA (multivariate analysis of variance) and are sometimes simply called MANOVA matrices.  The joint law of eigenvalues is explicitly known, but the large-$N$ limit is notoriously harder than the Gaussian Unitary and Wishart Ensembles.  In \cite{Collins2005}, the present first author made the following discovery which led to a new approach to the asymptotics of the ensemble: its joint law can be described by a product of randomly rotated projections, as follows.  (For the sake of making the statement simpler, we assume $a,b$ are such that $aN$ and $bN$ are integers.)

\begin{theorem}[\cite{Collins2005}, Theorem 2.2] \label{t.Collins.Jacobi} Let $J_{a,b}=J_{a,b}^N$ be an $N\times N$ Jacobi ensemble with parameters $a,b\ge1$.  Let $P,Q\in\M_{(a+b)N}$ be (deterministic) orthogonal projections with $\mathrm{rank}(P) = bN$ and $\mathrm{rank}(Q) = N$.  Let $U\in\U_{(a+b)N}$ be a random unitary matrix sampled from the Haar measure.  Then $Q U^\ast PU Q$, viewed as a random matrix in $\M_N$ via the unitary isomorphism $\M_N \cong Q\M_{(a+b)N}Q$, has the same distribution as $J_{a,b}$.  \end{theorem}

Given two closed subspaces $\mathbb{V},\mathbb{W}$ of a Hilbert space $\mathbb{H}$, if $P\colon\mathbb{H}\to\mathbb{V}$ and $Q\colon\mathbb{H}\to\mathbb{W}$ are the orthogonal projections, then the operator $QPQ$ is known as the {\em operator valued angle} between the two subspaces.  (Indeed, in the finite-dimensional setting, the eigenvalues of $QPQ$ are trigonometric polynomials in the principal angles between the subspaces $\mathbb{V}$ and $\mathbb{W}$.)  Thus, the law of the Jacobi ensemble records all the remaining information about the angles between two uniformly randomly rotated subspaces of fixed ranks.  These observations were used to make significant progress in understanding the Jacobi Ensemble in statistical applications (cf.\ \cite{Johnstone2008}), and to generalize many of these results to the full expected universality class (beyond Gaussian entries) in the limit (cf.\ \cite{ErdosFarrell2013}).

In terms of the large-$N$ limit: letting $\alpha = \frac{b}{a+b}$ and $\beta = \frac{1}{a+b}$, we have $\tr P = \alpha$ and $\tr Q = \beta$ fixed as $N$ grows, and therefore there are limit projections $p,q$ of these same traces. The chosen Haar distributed unitary matrices converge in noncommutative distribution to a Haar unitary operator $u$ freely independent from $p,q$, and so the empirical spectral distribution of $J_{a,b}$ converges to the law of $qu^\ast p u q$, which was explicitly computed in \cite{VDN} as an elementary example of free multiplicative convolution:
\[ \mathrm{Law}_{qu^\ast p u q} = (1-\min\{\alpha,\beta\})\delta_0 + \max\{\alpha+\beta-1,0\}\delta_1 + \frac{\sqrt{(r_+-x)(x-r_-)}}{2\pi x(1-x)}\1_{[r_-,r_+]}\,dx, \]
where $r_{\pm} = \alpha+\beta-2\alpha\beta\pm2\sqrt{\alpha\beta(1-\alpha)(1-\beta)}$.  Furthermore, it was shown in \cite{Johnstone2008} that the Jacobi Ensemble has a hard edge, the rate of convergence of the largest eigenvalue is $N^{-2/3}$ (as with the Gaussian Unitary and Wishart Ensembles), and the rescaled limit distribution of the largest eigenvalue is the Tracy-Widom law of \cite{TracyWidom}.

Simultaneously to these developments, Voiculescu \cite{Voiculescu1999} introduced {\em free liberation}.  Given two subalgebras $A,B$ of a $W^\ast$-probability space $(\mathscr{A},\tau)$ and a Haar unitary operator $u\in\mathscr{A}$ that is freely independent from $A,B$, the rotated subalgebra $u^\ast A u$ is freely independent from $B$.  If $(u_t)_{t\ge 0}$ is a free unitary Brownian motion freely independent from $A,B$, it is not generally true that $u_t^\ast A u_t$ is free from $B$ for any finite $t$ (in particular when $t=0$ we just have $A,B$), but since the (strong operator) limit as $t\to\infty$ of $u_t$ is a Haar unitary, this process ``liberates'' $A$ and $B$.  This concept was used to define several important regularized versions of measures associated to free entropy and free information theory, and to this days plays an important role in free probability theory.  The special case that $A,B$ are algebras generated by two projections has been extensively studied \cite{Demni2008,Demni2015,DemniHamdiHmidi2012,DemniHmidi2013,DemniZani2009,HiaiMiyamotoUeda2009,HiaiUeda2008,HiaiUeda2009}, as the best special case where one can hope to compute all quantities fairly explicitly.

In the first and third authors' paper \cite[Section 3.2]{CK2014}, the following was proved.

\begin{theorem}[\cite{CK2014}, Lemmas 3.2--3.6] \label{t.CK2014} Let $p,q$ be orthogonal projections with traces $\alpha,\beta$, and let $(u_t)_{t\ge 0}$ be a free unitary Brownian motion freely independent from $p,q$.  Let $\mu_t = \mathrm{Law}_{qu_t^\ast p u_t q}$.  Then
\[ \mu_t = (1-\min\{\alpha,\beta\})\delta_0 + \max\{\alpha+\beta-1,0\}\delta_1 + \widetilde{\mu}_t \]
where $\widetilde{\mu}_t$ is a positive measure (of mass $\min\{\alpha,\beta\}-\max\{\alpha+\beta-1,0\}$).  Let $I_1,I_2$ be two disjoint open subintervals of $(0,1)$.  If $\supp \widetilde{\mu}_{t_0}\subset I_1\sqcup I_2$ for some $t_0\ge 0$, then $\supp\widetilde{\mu}_t\subset I_1\sqcup I_2$ for $|t-t_0|$ sufficiently small; moreover, $\widetilde{\mu}_t(I_1)$ and $\widetilde{\mu}_t(I_2)$ do not vary with $t$ close to $t_0$.

If $\widetilde{\mu}_t$ has a continuous density on $(0,1)$ for $t>0$, and $x_{t_0}\in(0,1)$ is a boundary point of $\supp\widetilde{\mu}_{t_0}$, then for $|t-t_0|$ sufficiently small there is a $C^1$ function $t\mapsto x(t)$  with $x(t_0) = x_{t_0}$ so that $x(t)$ is a boundary point of $\supp\widetilde{\mu}_{t_0}$.

Finally, in the special case $\alpha=\beta=\frac12$, for all $t>0$, $\widetilde{\mu}_t$ possesses a continuous density which is analytic on the interior of its support.
\end{theorem}

\begin{remark} \begin{itemize} \item[(1)] It is expected that the final statement, regarding the existence of a continuous density, holds true for all $\alpha,\beta\in(0,1)$; at present time, this is only known for $\alpha=\beta=\frac12$.  Nevertheless, the ``islands stay separated'' result holds in general.
\item[(2)] Our method of proof of the regularity of $\widetilde{\mu}_t$ involved a combination of free probability, complex analysis, and PDE techniques.  In the preprint \cite{IzumiUeda}, the authors partly extended this framework beyond the $\alpha=\beta=\frac12$ case; but they were still not able to prove continuity of the measure.  They did, however, give a much simpler proof of the result in the case $\alpha=\beta=\frac12$: here, $\widetilde{\mu}_t$ can be described as the so-called Szeg\H{o} transform (from the unit circle to the unit interval) of the law of $v_0u_t$, where $v_0$ is the inverse Szeg\H{o} transform of the law of $qpq$.  Via this description, the regularity result is an immediate consequence of Theorem \ref{t.nu.t} above.
\item[(3)] Let us note that $\alpha=\beta=\frac12$ corresponds to $a=b=1$, meaning the ``square'' Jacobi ensemble.  This is, of course, the case that is least interesting to statisticians: in MANOVA problems the data sets are typically time series, where there are many more samples than detection sites, meaning that $a,b\ggg 1$.  In fact, it is generally questioned whether the Jacobi Ensemble is a realistic regime for real world applications, rather than building the Wishart Ensembles out of $N\times M$ Gaussian matrices where $\frac{M}{N}\to\infty$.
\end{itemize}
\end{remark}

Thus, it is natural to consider the corresponding finite-$t$ deformation of the Jacobi Ensemble.    The {\em matrix Jacobi process} $J_t^N$ associated to the projections $P^N,Q^N\in\M_N$, is given by
\begin{equation} \label{e.Jacobi.process} J^N_t = Q^N (U^N_t)^{\ast} P^N U^N_t Q^N \end{equation}
where $(U^N_t)_{t\ge 0}$ is a Brownian motion in $\U_N$.  (Typically $P^N,Q^N$ are deterministic; they may also be chosen randomly, in which case $U^N_t$ must be chosen independent from them.)  This is a diffusion process in $\M_N^{[0,1]}$: it lives a.s.\ in the space of matrices $M\in\M_N$ with $0\le M\le 1$ (i.e.\ $M$ is self-adjoint and has eigenvalues in $[0,1]$).  Note that the initial value is $J_0^N = Q^NP^NQ^N$, the operator-valued angle between the subspaces in the images of $P^N,Q^N$.  In particular, the Jacobi process records (through its eigenvalues) the evolution of the principal angles between two subspaces as they are continuously rotated by a unitary Brownian motion.

In the case $N=1$, the process \eqref{e.Jacobi.process} precisely corresponds to what is classically known as the Jacobi process: the Markov process on $[0,1]$ with generator $\mathscr{L} = x(x-1)\frac{\del^2}{\del x^2} - (cx+d)\frac{\del}{\del x}$, where $c=2\min\{\alpha,\beta\}-1$, $d=|\alpha-\beta|$.  This is where the name comes from, as the orthogonal polynomials associated to this Markov process are the Jacobi polynomials, cf.\ \cite{DemniZani2009}.

\begin{remark} Comparing to Theorem \ref{t.Collins.Jacobi}, we have now compressed the projections and the Brownian motion into $\M_N$ from the start.  We could instead formulate the process as in that theorem by choosing projections and Brownian motion in a larger space, which would have the effect of using a ``corner'' of a higher-dimensional Brownian motion instead of $U^N_t$.  While this makes a difference for the finite-dimensional distribution, it does not affect the large-$N$ behavior.  \end{remark}

This brings us to our main application.  First note that, from our main Theorem \ref{t.main.2}, the Jacobi process converges strongly.

\begin{corollary} Let $P^N,Q^N$ be deterministic orthogonal projections in $\M_N$, and suppose $\{P^N,Q^N\}$ converges strongly to $\{p,q\}$.  Let $(u_t)_{t\ge 0}$ be a free unitary Brownian motion freely independent from $p,q$.  Then for each $t\ge 0$ the Jacobi process marginal $J^N_t$ converges strongly to $j_t = qu_t^\ast p u_t q$.  What's more, if $f\in C[0,1]$ is any continuous test function, then $\|f(J^N_t)\|\to \|f(j_t)\|$ a.s.\ as $N\to\infty$.  \end{corollary}

\begin{proof} \label{c.strong.conv.Jacobi} The strong convergence statement is an immediate corollary to Theorem \ref{t.main.2}, with $A^N_1 = P^N$, $A^N_2 = Q^N$, and $n=2,k=1$.  The extension to continuous test functions beyond polynomials is then an elementary Weierstrass approximation argument.  \end{proof}

\begin{example} \label{eg.Jacobi} For fixed $k\in\N$, select two orthogonal projections $P,Q\in\M_k$.  Then define $P^N,Q^N\in\M_{kN}$ by $P^N = P\tensor I^N$ and $Q^N = Q\tensor I^N$.  (Here we are identifying $\M_k\tensor\M_N\cong \M_{kN}$ via the Kronecker product.)  If $F$ is a noncommutative polynomial in two variables, then
\[ F(P^N,Q^N) = F(P,Q)\tensor I^N \]
and it follows immediately that $\{P^N,Q^N\}$ converges strongly to $\{P,Q\}$ (i.e.\ the $W^\ast$-probability space can be taken to be $(\M_k,\tr)$).  Expanding this space to include a free unitary Brownian motion freely independent from $\{P,Q\}$ and setting $j_t = Qu_t^\ast P u_t Q$, Corollary \ref{c.strong.conv.Jacobi} yields that the Jacobi process $J^{kN}_t$ with initial value $Q^NP^NQ^N$ converges strongly to $j_t$.

Figure \ref{fig spec Jacobi} illustrates the eigenvalues of $J^{kN}_t$ with $k=4$, $N=100$, and initial projections given by
\[ P = \left[\begin{array}{cc} 0.2 & 0.4 \\ 0.4 & 0.8 \end{array}\right]\tensor \left[\begin{array}{cc} 1 & 0 \\ 0 & 0 \end{array}\right]  +  \left[\begin{array}{cc} 0.8 & 0.4 \\ 0.4 & 0.2 \end{array}\right]\tensor \left[\begin{array}{cc} 0 & 0 \\ 0 & 1 \end{array}\right], \qquad  Q = \left[\begin{array}{cc} 1 & 0 \\ 0 & 0 \end{array}\right]\tensor I_2 \]
which have been selected so that the initial operator-valued angle $QPQ$ has non-trivial eigenvalues $0.2$ and $0.8$; this therefore holds as well for $Q^NP^NQ^N$ for all $N$.  This implies that the subspaces $P^N(\C^{kN})$ and $Q^N(\C^{kN})$ have precisely two distinct principal angles.
\end{example}

\begin{figure}[htbp]
\begin{center}
\includegraphics[scale=0.208]{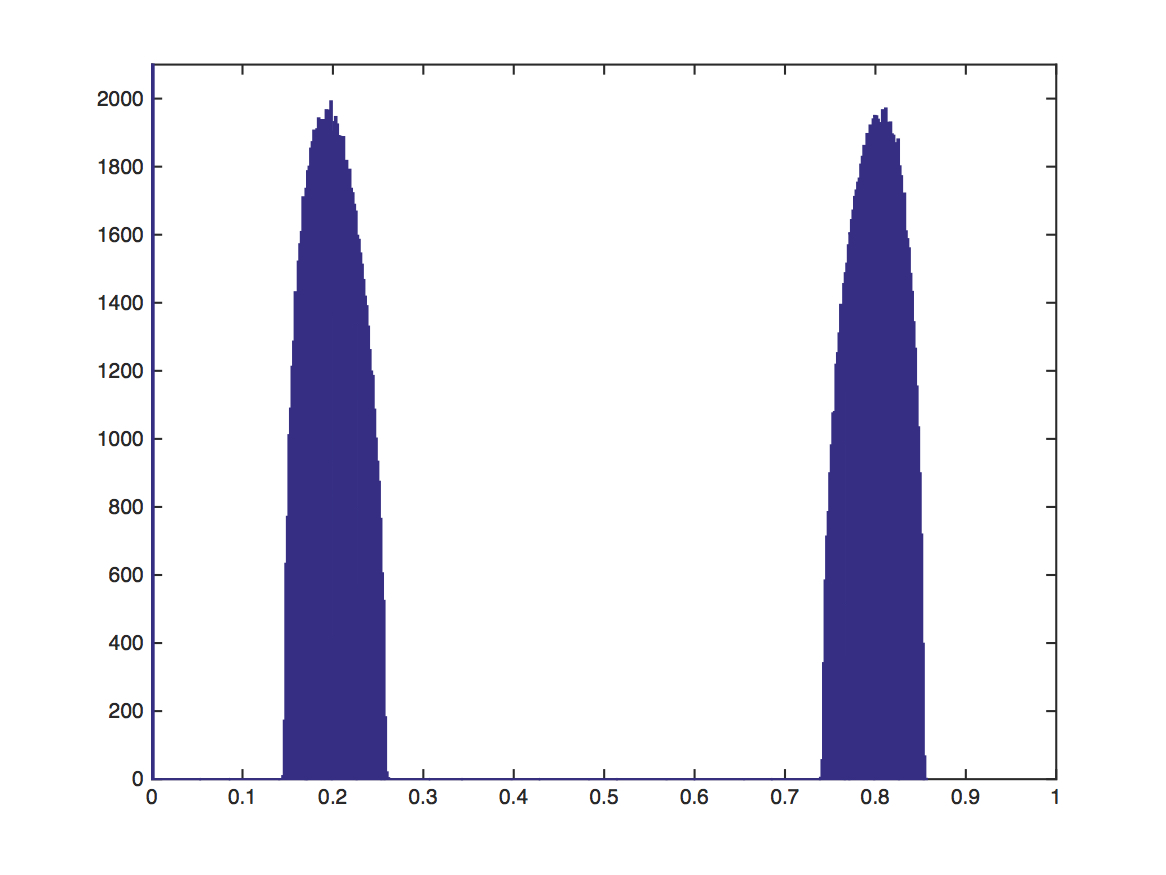}
\includegraphics[scale=0.208]{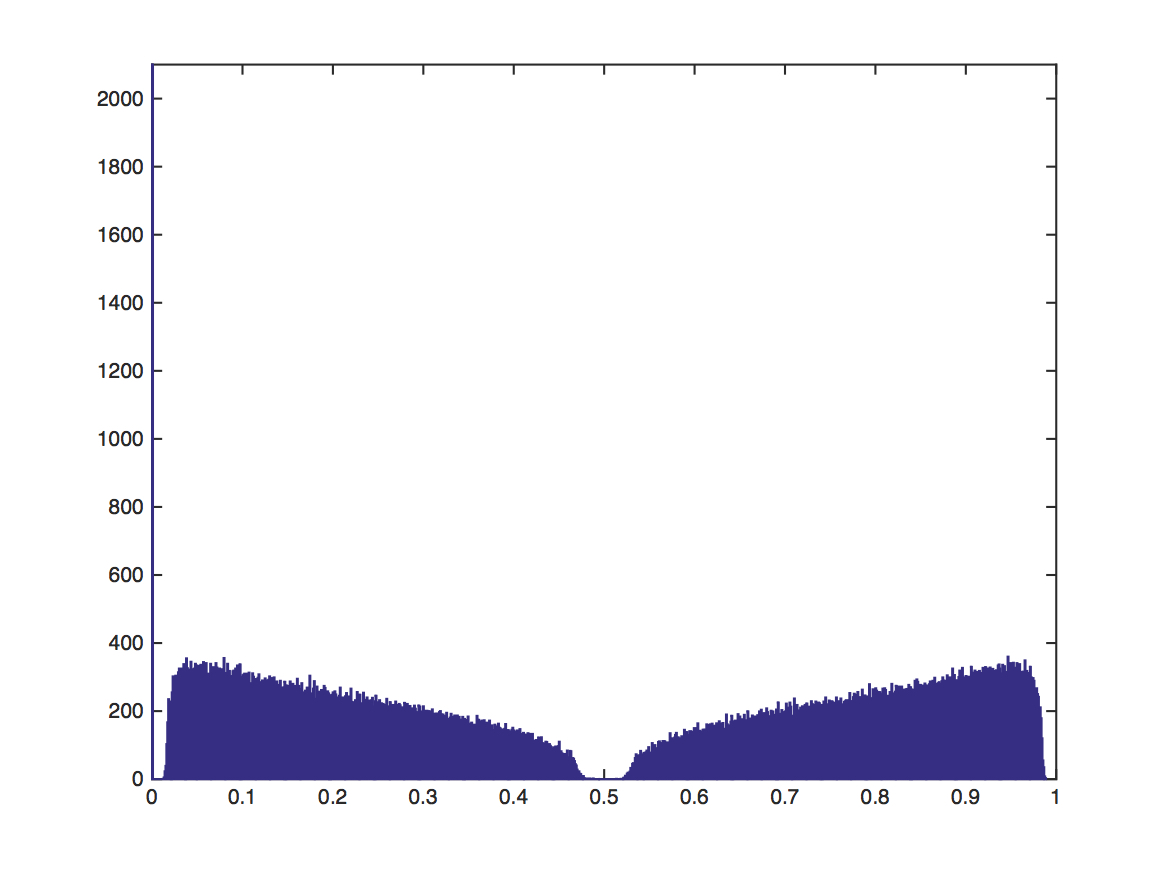}
\caption{\label{fig spec Jacobi} The spectral distribution of the Jacobi process $J^{kN}_t$ of Example \ref{eg.Jacobi} with $k=4$, $N=100$ and times $t=0.01$ (on the left) and $t=0.25$ (on the right).  The histograms were made with $1000$ trials each, yielding $4\times10^5$ eigenvalues sorted into $1000$ bins.}
\end{center}
\end{figure}

As is plainly visible in Figure \ref{fig spec Jacobi}, for small time, the eigenvalues (which are fixed trigonometric polynomials in the principal angles) stay close to their initial values.  That is: despite the fact that the diffusion's measure is fully supported on $\M_N^{[0,1]}$ for every $t,N>0$, the eigenvalues move with finite speed for all large $N$.  That is our final theorem.

\begin{theorem} \label{t.final} For each $N\ge 1$, let $(U^N_t)_{t\ge 0}$ be a Brownian motion on $\U_N$, let $\mathbb{V}^N$ and $\mathbb{W}^N$ be subspaces of $\C^N$, and suppose that the orthogonal projections onto these subspaces converge jointly strongly as $N\to\infty$.  Suppose there is a fixed finite set $\boldsymbol{\theta} = \{\theta_1,\ldots,\theta_k\}$ of angles so that all principal angles between $\mathbb{V}^N$ and $\mathbb{W}^N$ are in $\boldsymbol{\theta}$ for all $N$.  Then, for any open neighborhood $\mathscr{O}$ of $\boldsymbol{\theta}$, there is a $T>0$ so that, for $0\le t\le T$, it is a.s.\ true that for all large $N$, {\em all} principal angles between $U^N_t(\mathbb{V}^N)$ and $\mathbb{W}^N$ in $\mathscr{O}$.
\end{theorem}

\begin{proof} Let $P^N$ and $Q^N$ be the projections onto $\mathbb{V}^N$ and $\mathbb{W}^N$.  Then there is a fixed list $\boldsymbol{\lambda} = \{\lambda_1,\ldots,\lambda_k\}$ in $[0,1]$ so that all eigenvalues of $Q^NP^NQ^N$ are in $\boldsymbol{\lambda}$.  (The eigenvalues $\lambda_j$ are certain fixed trigonometric polynomials in $\boldsymbol{\theta}$).   Let $J^N_t$ be the Jacobi process associated to $P^N,Q^N$, and let $j_t$ be the associated large-$N$ limit.  By Corollary \ref{c.strong.conv.Jacobi}, for any $t\ge 0$ and any $f\in C[0,1]$, $\|f(J^N_t)\|\to \|f(j_t)\|$ a.s.\ as $N\to\infty$.

Applying this at time $t=0$, let $\lambda_i,\lambda_j\in\boldsymbol{\lambda}$ with $\lambda_i<\lambda_j$ such that no elements of $\boldsymbol{\lambda}$ are in the interval $(\lambda_i,\lambda_j)$.  Now let $f$ be a continuous bump function supported in $(\lambda_i,\lambda_j)$.  Then $f(J^N_0) = 0$, and it therefore follows that $\|f(j_0)\|=0$. As this holds for all bump function supported in $(\lambda_i,\lambda_j)$, it follows that $\mathrm{spec}(j_0)$ does not intersect $(\lambda_i,\lambda_j)$.  Thus $j_0$ has pure point spectrum precisely equal to $\boldsymbol{\lambda}$.

Now, fix any $\e>0$; by (induction on) Theorem \ref{t.CK2014}, for sufficiently small $t>0$, $\mathrm{spec}(j_t)$ is contained in $\boldsymbol{\lambda}_\e$ (the union of $\e$-balls centered at the points of $\boldsymbol{\lambda}$).  Now, suppose (for a contradiction) that, for some $N_0$, $J^{N_0}_t$ possesses an eigenvalue $\lambda\in (0,1)\setminus\boldsymbol{\lambda}_\e$.  Let $g$ be a bump function supported in $(0,1)\setminus\boldsymbol{\lambda}_\e$ that is equal to $1$ on a neighborhood of $\lambda$; then $\|g(J^{N_0}_t)\|\ge 1$.  But, by Corollary \ref{c.strong.conv.Jacobi}, we know $\|g(J^N_t)\|\to \|g(j_t)\| = 0$ a.s.\ as $N\to\infty$.  Thus, for all sufficiently large $N$, $\|g(J^N_t)\|<1$, which implies that $\lambda$ is not an eigenvalue.  As this holds for any point in $(0,1)\setminus\boldsymbol{\lambda}_\e$, it follows that $\mathrm{spec}(J^N_t)$ is almost surely contained in $\boldsymbol{\lambda}_\e$ for all sufficiently large $N$.

The result now follows from the fact that the principal angles between $U^N_t(\mathbb{V})$ and $\mathbb{W}$ are fixed continuous functions (trigonometric polynomials) in the eigenvalues of $J^N_t$. \end{proof}

\subsection*{Acknowledgments} This paper was initiated during a research visit of BC to TK in late 2012; substantial subsequent progress was made during a one-month visit of AD to BC at AIMR, Sendai during the fall of 2013.  We would like to express our great thanks to UC San Diego and AIMR for their warm hospitality, fruitful work environments, and financial support that aided in the completion of this work.  Finally, we would like to thank Ioana Dumitriu for helping to design the simulations that produced the figures.

\bibliographystyle{acm}
\bibliography{CDK}

\end{document}